\documentclass[a4paper,11pt,reqno]{amsart}
\usepackage{amssymb}
\usepackage{amsmath}
\usepackage{amsopn}
\usepackage{graphicx}
\usepackage{longtable}
\usepackage{mathrsfs}
\usepackage{amsmath}
\usepackage{amsfonts}
\usepackage{mathtools}
\usepackage{amsthm}
\usepackage{comment}
\usepackage{mwe}
\usepackage{tikz}
\usepackage{pstool}
\usepackage{psfrag}
\usepackage{epstopdf}
\usepackage{pstricks}
\usepackage{tikz}
\usepackage{marginnote}
\usetikzlibrary{shapes,backgrounds,decorations}
\usepackage{tikz-cd}
\usetikzlibrary{matrix,arrows,decorations.pathmorphing}
\usepackage{textcomp}     % for \textquotesingle macro
\usetikzlibrary{matrix,arrows}
\usepackage{todonotes}
\usepackage{marginnote}
\usepackage{lettrine}
\usepackage{booktabs}
\usepackage[all]{xy}
\usepackage{enumitem}
\PassOptionsToPackage{table}{xcolor}
\usepackage{xcolor,colortbl}
\definecolor{backcolour}{rgb}{0.95,0.95,0.95} 
\setlist{nosep}
\usepackage{latexsym}
\usepackage{eepic}
\usepackage{epsfig}
\usepackage{pst-node}
\usepackage{lipsum}
\usepackage{hyperref}
\hypersetup{
	colorlinks   = true, %Colours links instead of ugly boxes
	urlcolor     = blue, %Colour for external hyperlinks
	linkcolor    = blue, %Colour of internal links
	citecolor   = red %Colour of citations
}
\usepackage[scale=0.75]{geometry}
\usepackage{amscd}
\usepackage{mathrsfs}
\usepackage[english]{babel}

\usepackage[utf8]{inputenc}
\usepackage[%
style=alphabetic,sorting=nyt,
maxbibnames=10,maxcitenames=10,
sortcites=true,doi=false,maxalphanames=99,
firstinits=true,hyperref,backend=bibtex]{biblatex}
\renewbibmacro{in:}{%
	\ifentrytype{article}{}{\printtext{\bibstring{in}\intitlepunct}}}
\usepackage[babel]{csquotes}
\usepackage{guit}

\usepackage{multicol}

\DeclareFieldFormat
[article,inbook,incollection,inproceedings,patent,thesis,unpublished]
{title}{{#1\isdot}}
\DeclareFieldFormat{pages}{#1}
\usepackage{cancel}
\AtEveryBibitem{%
	\ifentrytype{book}{
		\clearfield{url}%
		\clearfield{pages}
		\clearfield{isbn}
		\clearfield{urldate}%
		\clearfield{review}%
		\clearfield{series}%%
	}{}
	\ifentrytype{article}{
		\clearfield{url}%
		\clearfield{issn}
		\clearfield{doi}
		\clearfield{isbn}
		\clearfield{urldate}%
		\clearfield{review}%
	}{}
	\ifentrytype{collection}{
		\clearfield{url}%
		\clearfield{urldate}%
		\clearfield{review}%
		\clearfield{series}%%
	}{}
	\ifentrytype{incollection}{
		\clearfield{url}%
		\clearfield{urldate}%
		\clearfield{review}%
		\clearfield{series}%%
	}{}
}

\usepackage[capitalise]{cleveref}
\numberwithin{equation}{section}

\newtheoremstyle{plain2}    % to uniformize styles with the paragraphs
{}            % ABOVESPACE (empty value is the same as default value)
{}            % BELOWSPACE (empty value is the same as degfault value)
{\itshape}    % BODYFONT  (\itshape, \bfseries, \normalfont)
{}            % INDENT (empty value is the same as 0pt, \parindent is the same as the standard indent for new paragraphs)
{\bfseries}   % HEADFONT
{.}           % HEADPUNCT
{5pt plus 1pt minus 1pt}  % HEADSPACE (leave 5pt plus 1pt minus 1pt)
{{\thmnumber{#1} \thmname{#2}{\thmnote{ (#3)}}}}          %CUSTOM-HEAD-SPEC
\theoremstyle{plain2}

\theoremstyle{plain}
\newtheorem{theorem}[equation]{Theorem}
\newtheorem{lemma}[equation]{Lemma}
\newtheorem{prop}[equation]{Proposition}

\newtheorem{defi}[equation]{Definition}

\newtheoremstyle{definition2}    % to uniformize styles with the paragraphs
{}
{}
{\normalfont}
{}
{\bfseries}
{.}
{5pt plus 1pt minus 1pt}
{{\thmnumber{#1} \thmname{#2}{\thmnote{#3}}}}
\theoremstyle{definition2}
\newtheorem{rem}[equation]{Remark}

\newtheorem{notation}[equation]{Notation}

\newtheorem*{theorem*}{Theorem}

\newcommand{\Z}{\ensuremath{\mathbb{Z}}}

\newcommand{\Q}{\ensuremath{\mathbb{Q}}}

\newcommand{\C}{\ensuremath{\mathbb{C}}}

\newcommand{\PP}{\ensuremath{\mathbb{P}}}

\renewcommand{\C}{\ensuremath{\mathbb{C}}}

\newcommand{\Aut}{\ensuremath{\mathrm{Aut}}}

\newcommand{\Sing}{\ensuremath{\mathrm{Sing}}}
\newcommand{\Bl}{\ensuremath{\mathrm{Bl}}}
\newcommand{\Stab}{\mathrm{Stab}}
\newcommand{\Fix}{\mathrm{Fix}}

\newcommand{\ord}{\mathrm{ord}}

\newcommand{\im}{\mathrm{Im}}

\newcommand{\SL}{\mathrm{SL}}
\newcommand{\Gr}{\mathrm{Gr}}

\newcommand{\codim}{\ensuremath{\mathrm{codim}}}

\newcommand{\rk}{\mathrm{rk}}

\newcommand{\PSL}{\mathrm{PSL}}
\newcommand{\PGL}{\mathrm{PGL}}
\newcommand{\diag}{\ensuremath{\mathrm{diag}}}
\newcommand{\om}{\ensuremath{\omega}}
\newcommand{\reg}{\mathrm{reg}}

\usepackage{listings}
\lstdefinelanguage{GAP}{%
  morekeywords={%
    Assert,Info,IsBound,QUIT,%
    TryNextMethod,Unbind,and,break,%
    continue,do,elif,%
    else,end,false,fi,for,%
    function,if,in,local,%
    mod,not,od,or,%
    quit,rec,repeat,return,%
    then,true,until,while%
  },%
  sensitive,%
  morecomment=[l]\#,%
  morestring=[b]",%
  morestring=[b]',%
}[keywords,comments,strings]

\usepackage{listings}

\definecolor{codegreen}{rgb}{0,0.6,0}
\definecolor{codegray}{rgb}{0.5,0.5,0.5}
\definecolor{codepurple}{rgb}{0.58,0,0.82}
\definecolor{backcolour}{rgb}{0.95,0.95,0.95} 

\lstdefinestyle{mystyle}{
    commentstyle=\color{codegreen},
    keywordstyle=\color{magenta},
    numberstyle=\tiny\color{codegray},
    stringstyle=\color{codepurple},
    basicstyle=\ttfamily\footnotesize,
    breakatwhitespace=false,         
    breaklines=true,                              
    keepspaces=true,                 
    showspaces=false,                
    showstringspaces=false,
    showtabs=false,                  
    tabsize=2
}
\title{Terminalizations of quotients\\of Fano varieties of lines on cubic fourfolds}
\author{Enrica Mazzon}
\address{Centre de Mathématiques Laurent Schwartz, Ecole Polytechnique, CNRS, Institut Polytechnique de Paris, Palaiseau, France}
\email{e.mazzon15@alumni.imperial.ac.uk}
\date{}

\begin{document}
\maketitle

\begin{abstract}
	We classify projective terminalizations of quotients of Fano varieties of lines on smooth cubic fourfolds by finite groups of symplectic automorphisms of the underlying cubic.
	We compute the second Betti number and the fundamental group of the regular locus. As a consequence, we identify two new deformation classes of four-dimensional irreducible holomorphic symplectic varieties with $b_2=4$ and simply connected regular locus.
\end{abstract}

\setcounter{tocdepth}{1} % This command is to see only part,chapters and sections in the table of contents
%\tableofcontents 

\section{Introduction}
Irreducible holomorphic symplectic (IHS) varieties are fundamental bulding blocks in the classification of varieties with trivial canonical class and possess exceptionally rich moduli and birational geometry. Beyond the classical families -- Hilbert schemes of points on a K3 surface, generalized Kummer varieties, and O’Grady examples in dimension six and ten -- constructing new deformation types of smooth IHS varieties has proven notoriously difficult.

In the singular setting, however, the situation is more flexible. A particularly fruitful strategy to produce examples of IHS varieties is to study projective terminalizations of quotients of smooth symplectic varieties by groups of symplectic automorphisms. 
This approach has been used in \cite{Fujiki1983,FuMenet2021,Menet2022,BertiniGrossiMauriMazzon2024}. Whereas \cite{Menet2022} gives a complete classification of Fujiki fourfolds (terminalizations of certain quotients of square of $\textrm{K3}$ surfaces), \cite{BertiniGrossiMauriMazzon2024} classifies all terminalizations of quotients of Hilbert schemes of $\textrm{K3}$ surfaces or generalized Kummer varieties by symplectic automorphisms induced from the underlying surface. Together, these works already yield at least 38 distinct deformation classes of four-dimensional irreducible symplectic varieties.

In this paper, we continue the classification program designed by Menet in \cite[Section 1.3]{Menet2022} by studying terminalizations of quotients of Fano varieties of lines on smooth cubic fourfolds by finite groups of symplectic automorphisms of the underlying cubic fourfold. Let $X$ be a smooth cubic fourfold and $F(X)$ its Fano variety of lines, which is a smooth four-dimensional IHS variety of $\textrm{K3}^{[2]}$-type. An automorphism of $X$ is called symplectic if the induced automorphism of $F(X)$ acts trivially on its holomorphic symplectic form.
Although $F(X)$ is deformation equivalent to the Hilbert scheme of two points on a $\textrm{K3}$ surface $S$, symplectic automorphisms of $F(X)$ induced by $X$ are in general not deformations of automorphisms of $S^{[2]}$ induced by $S$ previously studied in \cite{BertiniGrossiMauriMazzon2024,Menet2022}. Therefore, studying terminalizations of the quotients $F(X)/G$, where $G$ is a group of symplectic automorphisms of $X$, can provide genuinely new deformation classes of terminalizations.

\begin{theorem}\label{intro:thm}
Let $X$ be a smooth cubic fourfold.
For any finite group $G$ acting symplectically on $X$, the second Betti number and the fundamental group of the regular locus of a projective terminalization $Y$ of the quotient $F(X)/G$ are listed in Table \ref{table:term}.
\end{theorem}

In \cite{LazaZheng2022}, Laza and Zheng classified all possible symplectic automorphism groups of smooth cubic fourfolds.
Explicit equations for smooth cubic fourfolds realizing each such group as their symplectic automorphism group are known \cite{Mongardi2013, HoehnMason2019, YangYuZhu, Koike2024, Mongardi2013a, Adler1978, Fu2015}, and provide a natural starting point for the  classification in Theorem \ref{intro:thm}.
By \cite[Proposition 8.1]{BertiniGrossiMauriMazzon2024}, any terminalization in Table \ref{table:term} whose regular locus has nontrivial fundamental group is a quasi-étale quotient of a terminalization with simply connected regular locus, already appearing in the table. Thus, for purposes of classification, it suffices to focus on IHS varieties with simply connected regular locus.

To compute $b_2(Y)$, we establish the following formula.
\begin{theorem}[Theorem \ref{thm:b2}]
	Let $X$ be a smooth cubic fourfold, and $F(X)$ its Fano variety of lines. Let $G$ be a finite group acting symplectically on $X$. Let $q \colon F(X) \to F(X)/G$ be the quotient map, $Y$ a terminalization of $F(X)/G$, and $\Sigma$ the reduced singular locus of $F(X)/G$.
    
    For any $g \in G$, denote by $F_g \subset F(X)$ the unique component of the fixed locus of $g \in G$ of codimension $2$ (if any), and by $N_G(g)$ (respectively $C_G(g)$) the normaliser (respectively centraliser) of $g$ in $G$.
    Denote by
	\begin{itemize}
		\item $n_2$ the number of components 
		$q(F_g)$ in $\Sing(F(X)/G)$
		with $\ord(g)=2$;
		\item $n_{31}$ the number of components $q(F_g)$ in $\Sing(F(X)/G)$ with $g$ of order $3$ and such that $N_G(g)\setminus C_G(g)$ contains elements of even order;
		\item $n_{32}$ the number of components $q(F_g)$ in $\Sing(F(X)/G)$
		with $g$ of order $3$ and such that $N_G(g)\setminus C_G(g)$ contains no element of even order.
	\end{itemize}
	The following identity holds
	$$
	b_2(Y)= \rk(H^2(F(X),\Z)^G) + n_2 + n_{31} + 2n_{32}.
	$$
\end{theorem}
In this formula, only the number $n_2$ depends solely on the abstract structure of $G$. The quantities $n_{31}$ and $n_{32}$, instead, depend on the specific embedding of $G$ in $\PGL_6$; in Propositions \ref{prop:fixed loci codim2} and \ref{prop:b2_grouptheoretic} we give alternative descriptions of these terms, following \cite{Fu2015} which shows that a symplectic automorphism of $X$ of order $3$ fixes a locus of codimension 2 in $F(X)$ if and only if its normal form is $\diag(1,1,1,\omega,\omega, \omega).$
Such elements not only contribute to the formula for $b_2(Y)$, but also imply that the action of $G$ on $F(X)$ cannot arise from an action on the Hilbert scheme $S^{[2]}$ induced by an action on the $\textrm{K3}$ surface $S$, even though $S^{[2]}$ is deformation equivalent to $F(X)$. Indeed, the fixed locus of any induced symplectic automorphism of $S^{[2]}$ of order $3$ has no components of codimension $2$. Therefore, any group action in Table \ref{table:term} with nontrivial contribution $n_{31}$ or $n_{32}$ necessarily corresponds to an action that does not appear in the classification of \cite{BertiniGrossiMauriMazzon2024}.

Determining the rank of the invariant lattice $H^2(F(X),\Z)^G$ is a subtle task: although \cite{HoehnMason2019} lists possible ranks for all groups acting symplectically on $\textrm{K}3^{[2]}$-type manifolds, the actions themselves are not given and the same abstract group may appear with multiple ranks. In Section \ref{subsec:rk} we adopt various strategies to isolate the correct rank in each case.

Finally, using a numerical obstruction to deformation equivalence for IHS varieties arising from symplectic quotients, we identify genuinely new deformation types among our terminalizations.

\begin{theorem}[Theorem \ref{thm:b2 new}] \label{intro:thm new}
There are at least two new deformation classes of irreducible symplectic varieties with simply connected regular locus and $b_2=4$.
\end{theorem}
Such examples are rare: previously, only one deformation class with $b_2=4$ was known, constructed by Menet \cite{Menet2022} as a Fujiki fourfold. Our new classes arise from the maximal groups $G=A_7$ and $G=L_2(11)$. The alternating group $A_7$ acts in two distinct ways on two non-isomorphic smooth cubic fourfolds; whether the corresponding terminalizations are deformation equivalent remains open and likely requires a detailed analysis of their singularities. The same question arises in the cases where $b_2 >4$, since the numerical criterion alone does not allow to determine whether the terminalizations in Table \ref{table:term} give rise to new deformation classes.

\subsection*{Acknowledgements} The author is grateful to Kenji Koike, Radu Laza, Mirko Mauri, and Zhiwei Zheng for useful conversations and insightful remarks that contributed to the development of this work.

\section{Notation and Definitions} \label{sec:notation}
\subsection{Groups}\label{subsec:action} We denote by 
\begin{align*}
	C_n & \hspace{5mm} \text{cyclic group of order } n \\
	D_{2n}
	%=C_n \rtimes C_2 
	& \hspace{5mm} \text{dihedral group of order } 2n \\
	Q_n & \hspace{5mm} \text{generalised quaternion group of order }n \\
	S_n & \hspace{5mm} \text{group of permutations on $n$ elements}\\
	A_n & \hspace{5mm} \text{group of even permutations on $n$ elements} \\
	A_{m,n} &  \hspace{5mm} (S_m \times S_n) \cap A_{m+n} \\
	M_{10} & \hspace{5mm} \text{Mathieu group of degree 10} \\
	L_2(11) &  \hspace{5mm}  PSL_2(\mathbb{F}_{11}) \\
	%p^{1+2n}_{\pm }& \hspace{5mm} \text{extraspecial groups of order $p^{1+2n}$}\\
	\xi_m & \hspace{5mm} \text{primitive $m$-th root of unity}\\
	\om & \hspace{5mm} \text{primitive $3$-rd root of unity}
\end{align*}
Given a group $G$ and an element $g \in G$, we denote its centraliser and normaliser respectively by
\begin{align*}
	C_G(g) & \coloneqq \{ h \in G \,|\,  h^{-1}gh=g \},\\
	N_G(g)=N_G(\langle g \rangle) & \coloneqq  \{ h \in G \,|\, h^{-1}gh \in \langle g \rangle \}.
\end{align*}
All group actions in the paper are assumed to be faithful.

\subsection{IHS varieties} We recall some definitions.
\begin{defi}
Let $X$ be a normal variety, $j \colon X^{\mathrm{reg}} \hookrightarrow X$ the inclusion of the regular locus, and  $\Omega^{[p]}_{X}\coloneqq j_{*}\Omega^{p}_{X^{\mathrm{reg}}}$ the sheaf of reflexive $p$-forms on $X$.
A reflexive 2-form \[\omega_X \in H^0(X, \Omega^{[2]}_{X})=H^0(X^{\mathrm{reg}}, \Omega^{2}_{X^{\mathrm{reg}}})\] is \emph{symplectic} if its restriction to the regular locus of $X$ is closed non-degenerate.

\end{defi}
\begin{defi}
	A normal variety $X$ is \emph{symplectic} if \begin{itemize}
		\item it admits a symplectic form $\omega_X \in H^0(X, \Omega^{[2]}_{X})$;
		\item it has rational singularities.
	\end{itemize} 
	By \cite[Corollary 1.8]{KebekusSchnell2021}, this means that a holomorphic
	symplectic form $\omega_{X^{\mathrm{reg}}}$ on $X^{\mathrm{reg}}$ extends to a (possibly degenerate) holomorphic 2-form $\omega_{\widetilde{X}}$ on a resolution $\widetilde{X} \to X$. 
\end{defi}

\begin{defi}
	A symplectic compact K\"{a}hler
	%\footnote{We refer to \cite[\S3.3, p. 346]{G1962} or \cite[\S 2.3]{BL2022} for the notion of (possibly singular) compact K\"{a}hler space.} 
	variety $(X, \omega_X)$ is an \emph{irreducible holomorphic symplectic (IHS) variety} if for any finite quasi-\'{e}tale cover $g \colon X' \to X$ the exterior algebra of reflexive forms $H^0(X',\Omega^{[\bullet]}_{X'})$ on $X'$ is generated by the reflexive pullback $g^{[*]}\omega_X$.
\end{defi}

\section{Symplectic automorphisms of Fano varieties of lines} \label{sect:symplectic auto gps}

Let $V$ be a $6$-dimensional complex vector space and $\PP^5\coloneqq \PP(V)$. Given a smooth cubic fourfold $X \subset \PP^5$, we denote by $F(X)$ the Fano variety of lines on $X$. This is a $4$-dimensional smooth projective variety defined by
\[
F(X) \coloneqq \{ \ell \in \Gr(\PP^1, \PP^5) \,|\, \ell \subset X\}.
\]
By \cite{BeauvilleDonagi} $F(X)$ is an IHS variety, that we equip with the polarization $L$ given by the restriction of the Pl{\"u}cker line bundle on $\Gr(\PP^1, \PP^5)$.

\begin{defi}
	An automorphism $\phi$ of $F(X)$ is called
	\begin{itemize}
		\item polarized if it preserves the Pl{\"u}cker polarization $L$, i.e., $\phi^*L \simeq L$;
		\item symplectic if it acts as the identity on $H^{2,0}(F(X))$.
	\end{itemize}
\end{defi}

By \cite[Proposition 4]{Charles2012}, an automorphism of $F(X)$ is polarized if and only if it is induced by an automorphism of $X$ itself, hence in particular it comes from a linear automorphism of $\PP^5$, i.e.,
\[
\Aut(F(X), L) \simeq \Aut(X) \subset \PGL_6.
\]
We say that an automorphism $\phi$ of $X$ is symplectic if the induced automorphism of $F(X)$ is so, and we still denote it by $\phi$. We denote by $\Aut^s(X)$ the group of symplectic automorphisms of $X$; we say that $\Aut^s(X)$ is \emph{maximal} if it does not embed as a proper subgroup of the symplectic automorphism group of any other smooth cubic fourfold. 

For any group $G$, we denote by $\mathcal{M}_{G}$ the moduli space of smooth cubic fourfolds $X$ admitting $G \subseteq \Aut^s(X)$.
Two smooth cubic fourfolds $X$ and $X'$ in a same irreducible component of $\mathcal{M}_G$ are deformation equivalent as pairs $(X,G)$ and $(X',G)$, in the following sense: there exists a flat proper family $\mathcal{X}$ over a connected base with an algebraic continuous fiberwise action of $G$ on $\mathcal{X}$ such that two fibers $\mathcal{X}_t$ and $\mathcal{X}_{t'}$ are respectively isomorphic to $X$ and $X'$, and the action of $G$ restricts to the given actions on $X$ and $X'$.

Among the 34 possible symplectic automorphism groups of smooth cubic fourfolds classified in \cite{LazaZheng2022}, the associated moduli space is always equidimensional with at most two irreducible components, except in the case of the group $C_3$. Indeed, $\mathcal{M}_{C_3}$ has two irreducible components of dimension 8, corresponding respectively to the action of $\diag(1,1,1,1,\omega,\omega^2)$ and $\diag(1,1,\omega,\omega,\omega^2,\omega^2)$, and an additional irreducible component of dimension 2 where $C_3$ acts via $\diag(1,1,1,\omega,\omega,\omega)$. This two-dimensional component is in fact isomorphic to $\mathcal{M}_{G'}$ where $G' \simeq (C_3^3 \rtimes_3 C_3^2) \rtimes_3 C_2 \simeq 3_+^{1+4}\rtimes_3 C_2 $. This means that the generic smooth cubic fourfold admitting a symplectic automorphism with normal form $\diag(1,1,1,\omega,\omega,\omega)$ has full symplectic automorphism group $G'$; see \cite[Section 6.1]{Koike2024}.
For all other symplectic automorphism groups $G$, the equidimensionality of the moduli space $\mathcal{M}_G$ is not stated explicitly in \cite{Koike2024}, but follows from \cite{LazaZheng2022,HoehnMason2019}. Namely, the moduli space of smooth cubic fourfolds $X$ admitting $G\subseteq \Aut^s(X)$ with prescribed action, hence with fixed coinvariant lattice $H^4(X,\Z)_G \simeq S$, has dimension $20-\rk(S)$; see for instance \cite[Theorem 4.6]{LazaZheng2022}. The possible ranks of coinvariant lattices for every symplectic action of each group $G$ are listed in \cite[Table 12]{HoehnMason2019}.

\begin{prop} \label{prop:maximalG} \hfill
    \begin{enumerate}
        \item \label{prop:maximalGdef} Any smooth cubic fourfold $X$ together with a finite subgroup $G \subseteq \Aut^s(X)$ is deformation equivalent to a pair $(X',G)$, where $X'$ is a smooth cubic fourfold whose symplectic automorphism group $\Aut^s(X')$ is maximal.
        
        \item \label{prop:maximalGlist} A maximal symplectic automorphism group is isomorphic to one of the following groups:
        \begin{equation}\label{eq:maximalG}
		C_3^4 \rtimes A_6,\quad 
		A_7,\quad
		(C_3^3 \rtimes_3 C_3^2)\rtimes Q_8\footnote{The group $C_3^3 \rtimes_3 C_3^2$ is often denoted by $3_+^{1+4}$, one of the extraspecial groups of order $3^{1+4}$.},\quad
		M_{10},\quad
		L_2(11),\quad
		A_{3,5},\quad
		Q_8 \rtimes S_3.
	\end{equation}
    \end{enumerate}
\end{prop}
\begin{proof}
This follows from the classification of all finite groups that arise as $\Aut^s(X)$ for a smooth cubic fourfold $X$ \cite{LazaZheng2022}, and from the explicit determination, for each prescribed action of a group $G$, of equations defining smooth cubic fourfolds with $G \simeq \Aut^s(X)$ \cite{Koike2024}. In particular, part (\ref{prop:maximalGlist}) follows directly from \cite[Theorem 1.2]{LazaZheng2022}. 

For part (\ref{prop:maximalGdef}), let $X$ be a smooth cubic fourfold and let $G \subseteq \Aut^s(X)$ be a finite subgroup. Assume that $\mathcal{M}_{\Aut^s(X)}$ is irreducible. Then, for any maximal symplectic automorphism group $\Aut^s(X')$ containing $\Aut^s(X)$, the pair $(X,\Aut^s(X))$ is deformation equivalent to $(X',\Aut^s(X))$. It follows that the same deformation equivalence holds for $G$.

If instead $\mathcal{M}_{\Aut^s(X)}$ is not irreducible, we claim that any irreducible component of $\mathcal{M}_{\Aut^s(X)}$ contains a smooth cubic fourfold $X'$ such that $\Aut^s(X')$ is maximal, which implies (\ref{prop:maximalGdef}).
The cases in which $\mathcal{M}_{\Aut^s(X)}$ is not irreducible were determined in  \cite[Theorem 1.1 and Table 1]{Koike2024}: they occur precisely for the groups $$C_3,\,S_3,\,A_{3,3},\,D_{12},\,S_4,\,A_4,\,S_3^2,\,A_{4,3},\,A_5,\,C_3^2 \rtimes C_4,\,A_6,\,S_5,\,A_7,\,M_{10}.$$ In what follows, we prove the claim for the non-maximal symplectic automorphism groups appearing in this list.

The two 8-dimensional irreducible components of $\mathcal{M}_{C_3}$ correspond to the actions of $g_1=\diag(1,1,1,1,\omega,\omega^2)$, and $g_2=\diag(1,1,\omega, \omega,\omega^2, \omega^2)$, in normal form. Since the Fermat cubic $X_F$ has maximal symplectic automorphism group and admits the action of both matrices (see Appendix \ref{appendix}), each of these two components contains, respectively, $(X_F, \langle g_1 \rangle)$ and $(X_F, \langle g_2 \rangle)$. The remaining 2-dimensional irreducible component of $\mathcal{M}_{C_3}$ is isomorphic to an irreducible moduli space $\mathcal{M}_{G'}$, where $G' \simeq (C_3^3 \rtimes_3 C_3^2)$ is a symplectic automorphism group, thus the claim follows from the discussion above.
    
Both $D_{12}$ and $S_5$ are contained in the maximal symplectic automorphism group $A_7$, and $S_5$ has a unique subgroup isomorphic to $D_{12}$, up to conjugation. For the cubic fourfolds in one component of $\mathcal{M}_{D_{12}}$, the generator $l_1$ of the unique (up to conjugation) subgroup $C_6$ acts with normal form $\diag(-1,-1,1,1,\omega,\omega^2)$, while on the cubic fourfolds in the other component with normal form $\diag(1,-1,\omega,\omega,\omega^2,-\omega^2)$, by \cite[Section 4.1]{Koike2024}. 
The unique (up to conjugation) subgroup $C_6$ of $A_7$ acts on $X^1(A_7)$ by the permutation $(12)(34)(567)$, whose normal form coincides with that of $l_1$, and on $X^2(A_7)$ by the matrix $\diag(1,-1,\omega,\omega,\omega^2,-\omega^2)$; see Appendix \ref{appendix}. Therefore, the cubic fourfolds $X^1(A_7)$ and $X^2(A_7)$ lie in distinct irreducible components of $\mathcal{M}_{D_{12}}$ and of $\mathcal{M}_{S_5}$.

The group $H \in \{S_3,A_{3,3},\,A_4,\,A_{4,3},\,C_3^2 \rtimes C_4,\, S_4,\,A_5,\,A_6\}$ is contained in the maximal symplectic automorphism group $A_7$.
The action of $A_7$ on $X^1(A_7)$ admits elements with normal form $\diag(1,1,1,1,\omega,\omega^2)$, whereas the action on $X^2(A_7)$ contains no elements of this form. Hence, if the restriction of the action of $A_7$ on $X^1(A_7)$ to $H$ contains such an element (for instance a $3$-cycle), then $X^1(A_7)$ and $X^2(A_7)$ must lie in different irreducible components of $\mathcal{M}_H$. The groups $A_{3,3},\,A_4,\,A_{4,3},\, S_4,\,A_5$ and $A_6$ act by permutations on $X^1(A_7)$ and each contains a $3$-cycle. The same holds for $C_3^2 \rtimes C_4$ which contains a unique subgroup isomorphic to $A_{3,3}$.

The maximal symplectic automorphism group $C_3^4 \rtimes A_6$ contains the subgroups
\begin{align*}
    H_1 & \coloneqq \langle
    (12)(34),(13)(24),\diag(\omega,\omega^{2},1,1,1,1),\diag(1,1,\omega, \omega^{2},1,1) \rangle,\\
    H_2 & \coloneqq \langle
    (12)(56),(34)(56),\diag(\omega,\omega^{2},1,1,1,1),\diag(1,1,\omega, \omega^{2},1,1) \rangle,
\end{align*}
both isomorphic to $S_3^2$. The line $\ell = \{x_0=x_1=x_2=x_3=0\} \subset \PP^5$ is pointwise fixed by the action of $H_1$. Thus, on the Fermat cubic $X_F$, the intersection $X_F \cap \ell
%= \{[0:0:0:0:\lambda:-\lambda] \,|\, \lambda=\omega^i\}
$ consists of fixed points by $H_1$. 
Instead, the fixed locus by $H_2$ on $\PP^5$ is the single point $[0:0:0:0:1:1]$, which does not lie on $X_F$. Hence, the action of $H_2$ on $X_F$ has no fixed points. It follows that $(X_F,H_1)$ and $(X_F,H_2)$ belong to distinct irreducible components of $\mathcal{M}_{S_3^2}$.
\end{proof}
%In $A_7$ there exists element of order $3$ acting on $X^1(A_7)$ with normal form $\diag(1,1,1,1,\omega,\omega^2)$, for instance the permutation $(123)$, while none of the elements of $A_7$ of order $3$ acts on $X^2(A_7)$ with such normal form. 
%Given a subgroup $H \subseteq A_7$, the existence in $H$ of an element of order $3$ acting with normal form $\diag(1,1,1,1,\omega,\omega^2)$ on some cubic fourfold implies that $X^1(A_7)$ belongs to one irreducible component of $\mathcal{M}_H$, while $X^2(A_7)$ necessarily belongs to the other irreducible component. This occurs for all groups $A_{3,3},\,C_3^2 \rtimes C_4,\, S_4,\,A_5$ and $A_6$ above.
%This occurs for the group $A_{3,3}$ by \cite[Section 4.2]{Koike2024}; for $C_3^2 \rtimes C_4$ since it contains a unique subgroup $A_{3,3}$; for $S_4,A_5$ and $A_6$ since they contain the permutation $(123)$, By \cite[Section 4.2]{Koike2024}, $(123)$ acts on the cubic fourfolds in one component of $\mathcal{M}_{A_{3,3}}$, 

Proposition \ref{prop:maximalG} reduces the study of terminalizations of quotients of Fano varieties of lines on smooth cubic fourfolds, by groups of polarised symplectic automorphisms, to the case in which the smooth cubic fourfold has maximal symplectic automorphism group. For each maximal group $G$ listed in \eqref{eq:maximalG}, the equations of the cubic fourfolds in $\mathcal{M}_G$ were computed in \cite{LazaZheng2022,HoehnMason2019,Koike2024}. We refer to Appendix \ref{appendix} for more details and the explicit action of each group. 

\subsection{Codimension 2 fixed loci}
Terminalizations of quotients \(F(X)/G\), where $G$ is a finite group of polarized symplectic automorphisms, are nontrivial (i.e. not isomorphic to the quotient itself), precisely when $G$ contains elements whose fixed locus in $F(X)$ has components of codimension two; see for instance \cite[Claim 2.3.1]{Kollar2013} and \cite[Corollary 1]{Namikawa2001}. 
Such elements were studied in \cite{Fu2015}, where Fu established the following result.

\begin{prop}
%[Automorphisms with codimension 2 fixed loci]
	\label{prop:fixed loci codim2}
	A polarized symplectic automorphism $\phi$ of $F(X)$ fixes a locus of codimension $2$ if and only if
	\begin{enumerate}
		\item either $\ord(\phi)=2$, in which case it fixes 28 points and a $K3$ surface;
		\item or $\ord(\phi)=3$ and $\phi$ has normal form $\diag(1,1,1,\om,\om,\om)$, in which case it fixes an abelian surface.
	\end{enumerate}
\end{prop}
\begin{proof}
	See \cite[Theorem 1.1]{Fu2015}, and in particular \cite[\S 5]{Fu2015} and \cite[\S 7]{Camere2012} for the description of the fixed loci.
\end{proof}

\begin{rem}
    Let $G$ be a finite group acting symplectically on a smooth cubic fourfold $X$. Verifying if an element of $G$ satisfies condition (1) in Proposition \ref{prop:fixed loci codim2} is a purely group-theoretic question, in the sense that it depends only on the abstract group structure of $G$. Instead, condition (2) depends also on the embedding of $G$ in $\PGL_6$, hence on the action of the group on $X$.
\end{rem}

\begin{lemma}
Let $X$ be a smooth cubic fourfold. Denote by $N_3$ be the number of subgroups of $\Aut^s(X)$ generated by an element of order $3$ which fixes a codimension 2 locus in $F(X)$.
	\begin{enumerate}
		\item \label{item:N3=0} If $\Aut^s(X) \simeq A_7,\, M_{10},\, L_2(11),\, A_{3,5}$, or $ Q_8 \rtimes S_3$, then $N_3=0$.
		\item \label{item:N3<=1} If $\Aut^s(X) \simeq (C_3^3 \rtimes_3 C_3^2)\rtimes Q_8$, then $N_3 = 1$.
		\item \label{item:N3<=10} If $\Aut^s(X) \simeq C_3^4 \rtimes A_6$, then $N_3 = 10$.
	\end{enumerate}
\end{lemma}
\begin{proof}
The lemma follows from Proposition \ref{prop:fixed loci codim2} and the explicit description in Appendix \ref{appendix} of the actions of $\Aut^s(X)$ as subgroups of $\PGL_6$. 

In particular, for the actions of the groups in (\ref{eq:maximalG}) different from $C_3^4 \rtimes A_6$, we run a computer program to look for the subgroups generated by an element of order $3$ with normal form $\diag(1,1,1,\omega, \omega,\omega)$; see Appendix \ref{app:L3} for the GAP code.
It turns out that, in case \eqref{item:N3=0}, the characteristic polynomial of any element of $\Aut^s(X)$ of order $3$ is different from the characteristic polynomial $P_\phi$ of $\phi \coloneqq\diag(1,1,1,\om,\om,\om)$, hence $N_3=0$; in case \eqref{item:N3<=1} there is precisely one subgroup of order 3 in $(C_3^3 \rtimes_3 C_3^2)\rtimes Q_8$ generated by an element with characteristic polynomial $P_\phi$, hence $N_3 = 1$.

In case \eqref{item:N3<=10}, any element of $\Aut^s(X)$ is of the form
$$g=\diag(\om^{a_0},\om^{a_1},\ldots,\om^{a_4},\om^{a_5}) \circ \sigma$$ 
with $a_i \in \Z/{3\Z}$, $\sum_{i=0}^5 a_i=0$, $a_0=0$, and $\sigma \in A_6$. Writing $\sigma$ as a product of disjoint cycles $\sigma=\tau_1 \cdot \ldots \cdot \tau_m$, where $\tau_i=(j_1,\ldots, j_{k_i})$ has length $k_i$ and $\sum_{i=1}^m k_i=6$, the characteristic polynomial of $g$ is
$$
P_g(t) = \prod_{i=1}^m (t^{k_i}- \om^{a_{j_1}+ \ldots + a_{j_{k_i}}}).
$$
Then $P_g=P_\phi$ if and only if $\sigma=\mathrm{Id}$ and $\{a_0,a_1,\ldots,a_5\}=\{0,0,0,1,1,1\}$. Thus, there are at most 10 subgroups $\langle g \rangle$ in $\Aut^s(X)$ fixing a locus of codimension 2 in $F(X)$.
\end{proof}

\section{Second Betti number of a terminalization}
\begin{notation}
	\label{not:term}
	Let $F(X)$ be the Fano variety of lines on a smooth cubic fourfold $X \subset \PP^5$. Let $G$ be a finite subgroup of $\PGL_6$ acting symplectically on $X$. Denote by 
	\begin{itemize}
	    \item $q$ the quotient map $q \colon F(X) \to F(X)/G$,
	    \item $Y$ a terminalization of $F(X)/G$ and $p \colon Y \to F(X)/G$,
	    \item $\Sigma$ the reduced singular locus of $F(X)/G$.
	\end{itemize}
	For $g \in G$, denote by $F_g \subset F(X)$ the component of the fixed locus of $g$ of codimension $2$ (if any). By Proposition \ref{prop:fixed loci codim2}, if $F_g$ exists, then $\ord(g)\in \{2,3\}$ and $F_g$ is unique.
\end{notation}

We recall some properties. The quotient $F(X)/G$ has rational singularities \cite[Proposition 5.13]{KollarMori}, is symplectic, has canonical singularities \cite[Claim 2.3.1]{Kollar2013}, and is an IHS variety \cite[Proposition 3.16]{BertiniGrossiMauriMazzon2024}. 

A terminalization of $F(X)/G$ exists \cite[Corollary 1.4.3]{BirkarCasciniHaconEtAl}, and in particular it can be chosen projective and $\Q$-factorial. Moreover, it is unique up to isomorphism in codimension one \cite[Corollary 3.54]{KollarMori}, and is an IHS variety when it K{\"a}hler (or equivalently projective since it is always Moishezon) \cite[Proposition 12]{Schwald2020}. 

Following \cite[Theorem 2.3]{Kaledin2006},  denote by $(F(X)/G)^\circ$ the union of the smooth locus of $F(X)/G$ and the points of $q(F_g)$ with isotropy\footnote{The isotropy (group) of a point $z$ in $F(X)/G$ is the stabilizer of a point of the orbit $q^{-1}(z)$, up to conjugation.} $\langle g \rangle$, for some $g \in G$; the dissident locus of $F(X)/G$ is defined to be the complement of $(F(X)/G)^\circ$.

Away from the dissident locus of $F(X)/G$, a terminalization $Y$ of $F(X)/G$ is isomorphic to the blowup of the irreducible components of $\Sigma$ of codimension $2$ (see \cite[Proposition 3.7 and Corollary 5.9]{BertiniGrossiMauriMazzon2024}), i.e.
\[\begin{tikzcd}[column sep=tiny]
	&& {F(X) \supset F_g \hspace{115pt}} \\
	Y && {F(X)/G \supset \Sigma \coloneqq\Sing(F(X)/G) \supseteq q(F_g)} \\
	{Y^\circ} & {\Bl_{\Sigma \cap (F(X)/G)^\circ} (F(X)/G)^\circ} & {(F(X)/G)^\circ \hspace{135pt}}
	\arrow["q"', shift right=24, from=1-3, to=2-3]
	\arrow["p", from=2-1, to=2-3]
	\arrow["\normalsize{\simeq}"{description}, draw=none, from=3-1, to=3-2]
	\arrow[hook, from=3-1, to=2-1]
	\arrow[from=3-2, to=3-3]
	\arrow[hook, shift left=24, from=3-3, to=2-3]
\end{tikzcd}\]
and $\codim(Y \setminus Y^\circ) \geq 2$. Therefore, away from a locus of codimension $\geq 2$, a terminalization $Y$ is obtained from $F(X)/G$ by blowing up once the components $q(F_g)$ of $\Sigma$.

\begin{theorem}
	\label{thm:b2}
	In the notation of Section \ref{subsec:action} and \ref{not:term}, let $L=H^2(F(X),\Z)$. Among the components $q(F_g)$ of $\Sigma=\Sing(F(X)/G)$, we denote by
	\begin{itemize}
		\item $n_2$ the number of components with $\ord(g)=2$;
		\item $n_{31}$ the number of components with $g$ of order $3$ and such that $N_G(g)\setminus C_G(g)$ contains elements of even order;
		\item $n_{32}$ the number of components with $g$ of order $3$ and such that $N_G(g)\setminus C_G(g)$ contains no element of even order.
	\end{itemize}
	The following identity holds
	$$
	b_2(Y)= \rk(L^G) + n_2 + n_{31} + 2n_{32}.
	$$
\end{theorem}

\begin{proof}
	Let $\{E_i\}_{i \in I}$ be the set of exceptional divisors $E_i$ of the terminalization $p$. By the blowup formula, we have
	$$
	H^2(Y, \Q) \simeq H^2(F(X)/G, \Q) \oplus \bigoplus_{i \in I}H^0(E_i, \Q),
	$$
	thus $b_2(Y)=\rk(L^G)+|I|$. We reduce to compute the number of $p$-exceptional divisors. Away from the dissident locus, the divisors $E_i$ correspond to the exceptional divisors of the blowup of the components $q(F_g) \subseteq \Sigma$ of codimension $2$, for $g \in G$. In particular
	\begin{enumerate}
		\item when $\ord(g)=2$, then $q(F_g)$ has transversal $A_1$ singularities, and $p$ extracts one exceptional divisor from $q(F_g)$;
		\item when $\ord(g)=3$, then $q(F_g)$ has transversal $A_2$ singularities, and by Lemma \ref{lem:n3} below, $p$ extracts one exceptional divisor (respectively two exceptional divisors) from $q(F_g)$ if $N_G(g)\setminus C_G(g)$ contains (respectively does not contain) elements of even order.
	\end{enumerate} 
	We conclude the statement by proving Lemma \ref{lem:n3}.
\end{proof}

\begin{lemma}
	\label{lem:n3}
	Let $g \in G$ be an element of order $3$ and fixing a locus $F_g \subset F(X)$ of codimension $2$. Let $p' \colon Y' \to F(X)/G$ be the blowup of $F(X)/G$ along $q(F_g)$. 
	Away from the dissident locus, the number of exceptional divisors extracted by $p'$ is
	 \begin{itemize}
		\item 1 if $N_G(g)\setminus C_G(g)$ contains elements of even order,
		\item 2 otherwise.
	\end{itemize}
\end{lemma}
\begin{proof}
	We follow the argument of \cite[Lemma 5.3]{BertiniGrossiMauriMazzon2024}. The eigenvalues of the symplectic action of $g$ on the normal bundle $N_{F_g/F(X)}$ of $F_g$ in $F(X)$ are $\om$ and $\om^2$. Thus, $N_{F_g/F(X)}$ splits into the sum of two $\langle g \rangle$-equivariant line bundles 
	$$N_{F_g/F(X)} \simeq L_\om \oplus L_{\om^2}$$ 
	where $g$ acts on $L_\om$ by scaling by $\om$, and on $L_{\om^2}$ by scaling by $\om^2$. 
	
	We denote by $p_1$ the blowup of $F(X)$ along the $G$-orbit of $F_g$, and by $\tilde{E} \coloneqq p_1^{-1}(F_g) \simeq \PP(N_{F_g/F(X)})$ the exceptional divisor over $F_g$.
	Let $p_2$ be the blowup of the sections $s_\om \coloneqq \im(\PP(L_\om) \to \PP(L_\om \oplus L_{\om^2}))$ and $s_{\om^2} \coloneqq \im(\PP(L_{\om^2}) \to \PP(L_\om \oplus L_{\om^2}))$ of $\tilde{E} \to F_g$, with exceptional divisors $\tilde{E}_\om$ and $\tilde{E}_{\om^2}$.
	
	As in the proof of \cite[Lemma 5.3]{BertiniGrossiMauriMazzon2024}, the blowup $p'$ extracts two exceptional divisors unless there exists an element $\iota \in G$ such that
	$$\iota(\tilde{E}_\om)= \tilde{E}_{\om^2}.$$
	If there exists such $\iota$, then 
	\begin{itemize}
		\item[-] $\iota$ is of even order, otherwise $\tilde{E}_\om= \iota^{\ord(\iota)}(\tilde{E}_\om)=\iota(\tilde{E}_\om)=\tilde{E}_{\om^2}$ which is impossible;
		\item[-] $\iota$ preserves $F_g$, hence it is an element of $N_G(g)$ by Lemma \ref{lem:Stab Fg} below, but it does not commute with $g$ as otherwise $\iota(\tilde{E}_\om)= \tilde{E}_{\om}.$
	\end{itemize}
	Conversely, if  $N_G(g)\setminus C_G(g)$ contains an element $\iota$ of even order, then $\iota(\tilde{E}_\om)= \tilde{E}_{\om^2}.$
\end{proof}

\begin{lemma} \label{lem:Stab Fg}
	If $g \in G$ fixes a locus $F_g \subset F(X)$ of codimension $2$, then 
	\begin{align*}
		\Fix(F_g) & \coloneqq \{ h \in G \,|\, \forall x \in F_g, \, h(x)=x\} = \langle g \rangle\\
		\Stab(F_g) & \coloneqq \{ h \in G \,|\, h(F_g)=F_g\} = N_G(g).
	\end{align*}
\end{lemma}
\begin{proof}
	We clearly have $\langle g\rangle \subseteq \Fix(F_g)$. Conversely, let $\eta$ be the generic point of $F_g$. Then $\Fix(F_g)$ acts non-trivially on the normal space $N_{F_g/X, \eta} \simeq \C^2$. By the classification of finite subgroups of $\SL(2,\C)$ and Proposition \ref{prop:fixed loci codim2}, we conclude that $\Fix(F_g)$ is a cyclic group and necessarily generated by $g$.
	
	Let $h \in G$, then we have 
	$$
	h(F_g)=F_{g'} \iff F_{hgh^{-1}}=F_g' \iff hgh^{-1} \in \Fix(F_{g'}),
	$$
	hence by the first part of the proof, $\Stab(F_g)= \{h \in G \,|\, hgh^{-1} \in \langle g \rangle \} \eqqcolon N_G(g)$.
\end{proof}

\begin{prop}\label{prop:b2_grouptheoretic} In the Notation \ref{not:term} and of Theorem \ref{thm:b2}, we have that
	\begin{itemize}
		\item $n_2$ is the number of conjugacy classes of involutions of $G$;
		\item $n_{31}$ is the number of conjugacy classes of subgroups of $G$ of order $3$ whose generator $g$ has normal form $\diag(1,1,1,\om,\om,\om)$ and  $N_G(g)\setminus C_G(g)$ contains elements of even order;
		\item $n_{32}$ is the number of conjugacy classes of subgroups of $G$ of order $3$ whose generator $g$ has normal form $\diag(1,1,1,\om,\om,\om)$ and  $N_G(g)\setminus C_G(g)$ contains no elements of even order.
	\end{itemize}
\end{prop}

\begin{proof}
	This description of $n_2, n_{31}$, and $n_{32}$ is a direct consequence of Theorem \ref{thm:b2} and Proposition \ref{prop:fixed loci codim2}.
\end{proof}

\section{Terminalizations of quotients of Fano varieties of lines}\label{sec:term}

For any smooth cubic fourfold $X$ and any group $G$ of symplectic automorphisms of $X$, Table \ref{table:term} provides the second Betti number and the fundamental group of the regular locus of any projective terminalization $Y$ of $F(X)/G$. By Proposition \ref{prop:maximalG}, modulo deformation equivalences, it suffices to consider as $G$ the maximal symplectic automorphism groups and their subgroups. 
The invariants in Table \ref{table:term} are computed with the aid of the computer algebra system GAP; see Appendix \ref{app:Gap}. In particular, for each group $G$, we list
\begin{itemize}
	\item the group ID of $G$ as in \href{https://people.maths.bris.ac.uk/~matyd/GroupNames/index500.html}{GroupNames}, when available, or the order of $G$;
	\item an alias of $G$ as abstract group, and we write (\textasteriskcentered) when this is not available or not completely determined (for instance when the semidirect product structure is not specified);
	\item \(\rk \coloneqq 23 - \rk\, H^2(F(X))^G\), see \cref{subsec:rk}; 
	\item the numbers \(n_2,\, n_{31},\, n_{32}\) of codimension 2 components of the singular locus of \(F(X)/G\), as defined in Theorem \ref{thm:b2};
	\item \(b_2(Y)  = 23 - \rk + n_2 + n_{31} + 2n_{32} \), see Theorem \ref{thm:b2};
	\item the fundamental group \(\pi_1(Y^{\mathrm{reg}})\simeq G/N\) of the regular locus of $Y$, where $N$ is the subgroup generated by elements of $G$ fixing a codimension $2$ locus, see \cite[Proposition 8.1]{BertiniGrossiMauriMazzon2024}.
\end{itemize}
	Moreover, in \cref{table:term} we identify rows corresponding to subgroups $G_1$ and $G_2$ of the same maximal symplectic automorphism group such that either $G_1$ and $G_2$ are conjugate subgroups or they give the same string of invariants $[G_i, \rk, n_2, n_{31}, n_{32}, b_2, \pi_1]$.

\begin{center}
\begin{longtable}[]{l|l|l|l|l|l} 
	\caption{Terminalizations of \(F(X)/G\)}
	\label{table:term}
	\endfirsthead
	\endhead
		\multicolumn{6}{c}{ \(G \subseteq A_{3,5}\) } \\ 
		ID          & \(G\)          & \textrm{rk}   & \(n_2\)  & \(b_2\)& \(\pi_1(Y^\reg)\)                    \\ \hline
		(2,1)       & \(C_2\)                          & 8  & 1        & 16 & \{1\}               \\ \hline
		(4,2)       & \(C^2_2\)                        & 12 & 3        & 14 & \{1\}               \\ \hline
		(4,1)       & \(C_4\)                          & 14 & 1        & 10 & \(C_2\)             \\ \hline
		(6,1)       & \(S_3\)                          & 14 & 1        & 10 & \{1\}               \\ \hline
		\rowcolor{backcolour}(6,2)       & \(C_6\)                          & 16 & 1        & 8  & \(C_3\)             \\ \hline
		(8,3)       & \(D_{8}\)                        & 15 & 3        & 11 & \{1\}               \\ \hline
		(10,1)      & \(D_{10}\)                       & 16 & 1        & 8  & \{1\}               \\ \hline
		(12,3)      & \(A_4\)                          & 16 & 1        & 8  & \(C_3\)             \\ \hline
		(12,5)      & \(C_2 \times C_6\)               & 18 & 3        & 8  & \(C_3\)             \\ \hline
		\rowcolor{backcolour}(12,1)      & \(C_3 \rtimes C_4\)              & 18 & 1        & 6  & \(S_3\)             \\ \hline
		(12,4)      & \(D_{12}\)                       & 16 & 3        & 10 & \{1\}               \\ \hline
		\rowcolor{backcolour}(18,3)      & \(C_3 \times S_3\)               & 18 & 1        & 6  & \(C_3\)             \\ \hline
		(18,4)      & \(C_3 \rtimes S_3 \simeq A_{3,3}\)              & 16 & 1        & 8  & \{1\}               \\ \hline
		(20,3)      & \(F_5\)                          & 18 & 1        & 6  & \(C_2\)             \\ \hline
		(24,12)     & \(S_4\)                          & 17 & 2        & 8  & \{1\}               \\ \hline
		(24,8)      & \(C_3 \rtimes D_8\)              & 18 & 3        & 8  & \{1\}               \\ \hline
		(30,2)      & \(C_3 \times D_{10}\)            & 20 & 1        & 4  & \(C_3\)             \\ \hline
		(36,11)     & \(C_3 \times A_4\)               & 18 & 1        & 6  & \(C^2_3\)           \\ \hline
		(36,10)     & \(S^2_3\)                        & 18 & 3        & 8  & \{1\}               \\ \hline
		(60,5)      & \(A_5\)                          & 18 & 1        & 6  & \{1\}               \\ \hline
		(60,7)      & \(C_3 \rtimes F_5\)              & 20 & 1        & 4  & \(S_3\)             \\ \hline
		(72,43)     & \(C_3 \rtimes S_4\)              & 18 & 2        & 7  & \{1\}               \\ \hline
		(120,34)    & \(S_5\)                          & 19 & 2        & 6  & \{1\}               \\ \hline
		(180,19)    & \(C_3 \times A_5\)               & 20 & 1        & 4  & \(C_3\)             \\ \hline
		(360,120)   & \(A_{3,5}\)                      & 20 & 2        & 5  & \{1\}               \\ 
		\multicolumn{6}{c}{} \\ 
		\multicolumn{6}{c}{\(G \subseteq L_2(11)\) } \\ 
		%&                                  &    &          &    &                     \\ \hline 
		ID          & \(G\)          & \textrm{rk}   & \(n_2\)  & \(b_2\)& \(\pi_1(Y^\reg)\)                    \\ \hline
		(2,1)       & \(C_2\)                          & 8  & 1        & 16 & \{1\}               \\ \hline
		(4,2)       & \(C^2_2\)                        & 12 & 3        & 14 & \{1\}               \\ \hline
		(6,1)       & \(S_3\)                          & 14 & 1        & 10 & \{1\}               \\ \hline
		\rowcolor{backcolour}(6,2)       & \(C_6\)                          & 16 & 1        & 8  & \(C_3\)             \\ \hline
		(10,1)      & \(D_{10}\)                       & 16 & 1        & 8  & \{1\}               \\ \hline
		(12,3)      & \(A_4\)                          & 16 & 1        & 8  & \(C_3\)             \\ \hline
		(12,4)      & \(D_{12}\)                       & 16 & 3        & 10 & \{1\}               \\ \hline
		(60,5)      & \(A_5\)                          & 18 & 1        & 6  & \{1\}               \\ \hline
		(660,13)    & \(L_2(11)\)                      & 20 & 1        & 4  & \{1\}               \\ \multicolumn{6}{c}{} \\ 
		\multicolumn{6}{c}{\(G \subseteq M_{10}\) } \\
		%&                                  &    &          &    &                     \\ \hline 
		ID          & \(G\)          & \textrm{rk}   & \(n_2\)  & \(b_2\)& \(\pi_1(Y^\reg)\)                    \\ \hline
		(2,1)       & \(C_2\)                          & 8  & 1        & 16 & \{1\}               \\ \hline
		(4,2)       & \(C^2_2\)                        & 12 & 3        & 14 & \{1\}               \\ \hline
		(4,1)       & \(C_4\)                          & 14 & 1        & 10 & \(C_2\)             \\ \hline
		(6,1)       & \(S_3\)                          & 14 & 1        & 10 & \{1\}               \\ \hline
		(8,1)       & \(C_8\)                          & 18 & 1        & 6  & \(C_4\)             \\ \hline
		(8,4)       & \(Q_8\)                          & 17 & 1        & 7  & \(C^2_2\)           \\ \hline
		(8,3)       & \(D_{8}\)                        & 15 & 3        & 11 & \{1\}               \\ \hline
		(10,1)      & \(D_{10}\)                       & 16 & 1        & 8  & \{1\}               \\ \hline
		(12,3)      & \(A_4\)                          & 16 & 1        & 8  & \(C_3\)             \\ \hline
		(16,8)      &   \(Q_8 \rtimes C_2\)                               & 18 & 2        & 7  & \(C_2\)             \\ \hline
		(18,4)      & \(C_3 \rtimes S_3 \simeq A_{3,3}\)              & 16 & 1        & 8  & \{1\}               \\ \hline
		(20,3)      & \(C_5 \rtimes C_4\)                          & 18 & 1        & 6  & \(C_2\)             \\ \hline
		(24,12)     & \(S_4\)                          & 17 & 2        & 8  & \{1\}               \\ \hline
		(36,9)      & \(C^2_3 \rtimes C_4\)            & 18 & 1        & 6  & \(C_2\)             \\ \hline
		(60,5)      & \(A_5\)                          & 18 & 1        & 6  & \{1\}               \\ \hline
		(72,41)     &  \(C_3^2 \rtimes Q_8\)                                & 19 & 1        & 5  & \(C^2_2\)           \\ \hline
		(360,118)   & \(A_6\)                          & 19 & 1        & 5  & \{1\}               \\ \hline
		(720,765)   & \(M_{10}\)                       & 20 & 1        & 4  & \(C_2\)             \\ \multicolumn{6}{c}{} \\ 
		\multicolumn{6}{c}{\(G \subseteq (C_3^3 \rtimes_3 C_3^2)\rtimes Q_8\) } \\ 
		%&                                  &    &          &    &                     \\ \hline 
		ID          & \(G\)          & \textrm{rk}   & \((n_2,n_{31},n_{32})\)  & \(b_2\)& \(\pi_1(Y^\reg)\)                    \\ \hline
		(2,1)       & \(C_2\)                          & 8  & (1,0,0)  & 16 & \{1\}               \\ \hline
		\rowcolor{backcolour}(3,1)       & \(C_3\)                          & 18 & (0,0,1)  & 7  & \{1\}               \\ \hline
		(4,1)       & \(C_4\)                          & 14 & (1,0,0)  & 10 & \(C_2\)             \\ \hline
		\rowcolor{backcolour}(6,2)       & \(C_6\)                          & 18 & (1,0,1)  & 8  & \{1\}               \\ \hline
		(6,1)       & \(S_3\)                          & 14 & (1,0,0)  & 10 & \{1\}               \\ \hline
		(8,4)       & \(Q_8\)                          & 17 & (1,0,0)  & 7  & \(C^2_2\)           \\ \hline
		\rowcolor{backcolour}(9,2)       & \(C^2_3\)                        & 18 & (0,0,1)  & 7  & \(C_3\)             \\ \hline
		\rowcolor{backcolour}(12,1)      & \(C_3 \rtimes C_4\)              & 19 & (1,1,0)  & 6  & \(C_2\)             \\ \hline
		(12,2)      & \(C_{12}\)                       & 20 & (1,0,1)  & 6  & \(C_2\)             \\ \hline
		(18,4)      & \(C_3 \rtimes S_3 \simeq A_{3,3}\)              & 16 & (1,0,0)  & 8  & \{1\}               \\ \hline
		\rowcolor{backcolour}(18,3)      & \(C_3 \times S_3\)               & 18 & (1,0,1)  & 8  & \{1\}               \\ \hline
		(24,4)      &  \(C_3 \rtimes Q_8\)                                & 20 & (1,1,0)  & 5  & \(C^2_2\)           \\ \hline
		\rowcolor{backcolour}(27,5)      & \(C^3_3\)                        & 18 & (0,0,1)  & 7  & \(C^2_3\)           \\ \hline
		(27,3)      & \(C_3^2 \rtimes C_3\)                         & 18 & (0,0,1)  & 7  & \(C^2_3\)           \\ \hline
		(36,9)      & \(C^2_3 \rtimes C_4\)            & 18 & (1,0,0)  & 6  & \(C_2\)             \\ \hline
		\rowcolor{backcolour}(54,13)     & \(C_3 \times C_3 \rtimes S_3\)   & 18 & (1,0,1)  & 8  & \{1\}               \\ \hline
		(54,8)      &   \(C_3^2 \rtimes_2 S_3\)                               & 18 & (1,0,1)  & 8  & \{1\}               \\ \hline
		(72,41)     &   \(C_3^2 \rtimes Q_8\)                               & 19 & (1,0,0)  & 5  & \(C^2_2\)           \\ \hline
		(81,12)     &    \(C_3^2 \rtimes C_3^2\)                              & 18 & (0,0,1)  & 7  & \(C^3_3\)           \\ \hline
		\rowcolor{backcolour}(108,37)    & \(C^3_3 \rtimes_2 C_4\)          & 19 & (1,1,0)  & 6  & \(C_2\)             \\ \hline
		(108,36)    & \(C_3 \times C^2_3 \rtimes C_4\) & 20 & (1,0,1)  & 6  & \(C_2\)             \\ \hline
		(108,15)    &    \((C_3^2 \rtimes C_3) \rtimes C_4\)                              & 20 & (1,0,1)  & 6  & \(C_2\)             \\ \hline
		(162,46)    &   \(C_3^3 \rtimes_6 S_3\)                               & 18 & (1,0,1)  & 8  & \{1\}               \\ \hline
		(216,161)   & \(C^3_3 \rtimes_2 Q_8\)          & 20 & (1,1,0)  & 5  & \(C^2_2\)           \\ \hline
		(243,65)    &     \(C_3^3 \rtimes_3 C_3^2\)                             & 18 & (0,0,1)  & 7  &                     \\ \hline
		(486,249)   &    \((C_3^3 \rtimes_3 C_3^2) \rtimes_3 C_2\)                              & 18 & (1,0,1)  & 8  & \{1\}               \\ \hline
		(972,777)   &     \( ((C_3 \times (C_3^2 \rtimes C_3)) \rtimes C_3) \rtimes C_4 \) (\textasteriskcentered)                                               & 20 & (1,0,1)  & 6  & \(C_2\)             \\ \hline
		(972,776)   &  \( ((C_3 \times (C_3^2 \rtimes C_3)) \rtimes C_3) \rtimes C_4 \) (\textasteriskcentered)                                          & 19 & (1,1,0)  & 6  & \(C_2\)             \\ \hline
		(1944,3559) & \((C_3^3 \rtimes_3 C_3^2)\rtimes Q_8\)                & 20 & (1,1,0)  & 5  & \(C^2_2\)           \\ \multicolumn{6}{c}{} \\ 
		\multicolumn{6}{c}{\(G \subseteq A_7\) } \\ 
		%&                                  &    &          &    &                     \\ \hline 
		ID          & \(G\)          & \textrm{rk}   & \(n_2\)  & \(b_2\)& \(\pi_1(Y^\reg)\)                    \\ \hline
		(2,1)       & \(C_2\)                          & 8  & 1        & 16 & \{1\}               \\ \hline
		(4,2)       & \(C^2_2\)                        & 12 & 3        & 14 & \{1\}               \\ \hline
		(4,1)       & \(C_4\)                          & 14 & 1        & 10 & \(C_2\)             \\ \hline
		\rowcolor{backcolour}(6,2)       & \(C_6\)                          & 16 & 1        & 8  & \(C_3\)             \\ \hline
		(6,1)       & \(S_3\)                          & 14 & 1        & 10 & \{1\}               \\ \hline
		(8,3)       & \(D_{8}\)                        & 15 & 3        & 11 & \{1\}               \\ \hline
		(10,1)      & \(D_{10}\)                       & 16 & 1        & 8  & \{1\}               \\ \hline
		(12,3)      & \(A_4\)                          & 16 & 1        & 8  & \(C_3\)             \\ \hline
		(12,5)      & \(C_2 \times C_6\)               & 18 & 3        & 8  & \(C_3\)             \\ \hline
		\rowcolor{backcolour}(12,1)      & \(C_3 \rtimes C_4\)              & 18 & 1        & 6  & \(S_3\)             \\ \hline
		(12,4)      & \(D_{12}\)                       & 16 & 3        & 10 & \{1\}               \\ \hline
		(18,4)      & \(C_3 \rtimes S_3 \simeq A_{3,3}\)              & 16 & 1        & 8  & \{1\}               \\ \hline
		(20,3)      & \(C_5 \rtimes C_4\)                          & 18 & 1        & 6  & \(C_2\)             \\ \hline
		(24,12)     & \(S_4\)                          & 17 & 2        & 8  & \{1\}               \\ \hline
		(24,8)      & \(C_3 \rtimes D_8\)              & 18 & 3        & 8  & \{1\}               \\ \hline
		(36,11)     & \(C_3 \times A_4\)               & 18 & 1        & 6  & \(C^2_3\)           \\ \hline
		(36,9)      & \(C^2_3 \rtimes C_4\)            & 18 & 1        & 6  & \(C_2\)             \\ \hline
		(60,5)      & \(A_5\)                          & 18 & 1        & 6  & \{1\}               \\ \hline
		(72,43)     & \(C_3 \rtimes S_4\)              & 18 & 2        & 7  & \{1\}               \\ \hline
		(120,34)    & \(S_5\)                          & 19 & 2        & 6  & \{1\}               \\ \hline
		(168,42)    &   \(L_2(7)\)                               & 19 & 1        & 5  & \{1\}               \\ \hline
		(360,118)   & \(A_6\)                          & 19 & 1        & 5  & \{1\}               \\ \hline
		2520        & \(A_7\)                          & 20 & 1        & 4  & \{1\}               \\ \multicolumn{6}{c}{} \\ 
		\multicolumn{6}{c}{\(G \subseteq C_3^4 \rtimes A_6\) } \\ 
		%&                                  &    &          &    &                     \\ \hline 
		ID          & \(G\)          & \textrm{rk}   & \((n_2,n_{31},n_{32})\)  & \(b_2\)& \(\pi_1(Y^\reg)\)                    \\ \hline
		(2,1)       & \(C_2\)                          & 8  & (1,0,0)  & 16 & \{1\}               \\ \hline
		\rowcolor{backcolour}(3,1)       & \(C_3\)                          & 18 & (0,0,1)  & 7  & \{1\}               \\ \hline
		(4,2)       & \(C^2_2\)                        & 12 & (3,0,0)  & 14 & \{1\}               \\ \hline
		(4,1)       & \(C_4\)                          & 14 & (1,0,0)  & 10 & \(C_2\)             \\ \hline
		(6,1)       & \(S_3\)                          & 14 & (1,0,0)  & 10 & \{1\}               \\ \hline
		\rowcolor{backcolour}(6,2)       & \(C_6\)                          & 16 & (1,0,0)  & 8  & \(C_3\)             \\ \hline
		\rowcolor{backcolour}(6,2)       & \(C_6\)                          & 18 & (1,0,1)  & 8  & \{1\}               \\ \hline
		(8,3)       & \(D_{8}\)                        & 15 & (3,0,0)  & 11 & \{1\}               \\ \hline
		\rowcolor{backcolour}(9,2)       & \(C^2_3\)                        & 18 & (0,0,1)  & 7  & \(C_3\)             \\ \hline
		\rowcolor{backcolour}(9,2)       & \(C^2_3\)                        & 20 & (0,0,2)  & 7  & \{1\}               \\ \hline
		(9,1)       & \(C_9\)                          & 20 & (0,0,1)  & 5  & \(C_3\)             \\ \hline
		(10,1)      & \(D_{10}\)                       & 16 & (1,0,0)  & 8  & \{1\}               \\ \hline
		(12,5)      & \(C_2 \times C_6\)               & 18 & (3,0,0)  & 8  & \(C_3\)             \\ \hline
		(12,3)      & \(A_4\)                          & 16 & (1,0,0)  & 8  & \(C_3\)             \\ \hline
		(12,4)      & \(D_{12}\)                       & 16 & (3,0,0)  & 10 & \{1\}               \\ \hline
		\rowcolor{backcolour}(12,1)      & \(C_3 \rtimes C_4\)              & 18 & (1,0,0)  & 6  & \(S_3\)             \\ \hline
		\rowcolor{backcolour}(12,1)      & \(C_3 \rtimes C_4\)              & 19 & (1,1,0)  & 6  & \(C_2\)             \\ \hline
		(18,4)      & \(C_3 \rtimes S_3 \simeq A_{3,3}\)              & 16 & (1,0,0)  & 8  & \{1\}               \\ \hline
		\rowcolor{backcolour}(18,3)      & \(C_3 \times S_3\)               & 18 & (1,0,1)  & 8  & \(C_3\)              \\ \hline
		\rowcolor{backcolour}(18,3)      & \(C_3 \times S_3\)               & 18 & (1,0,0)  & 6  & \(C_3\)             \\ \hline
		\rowcolor{backcolour}(18,3)      & \(C_3 \times S_3\)               & 20 & (1,0,1)  & 6  & 
		\(\{1\}\)           \\ \hline
		(18,5)      & \(C_3 \times C_6\)               & 20 & (1,0,2)  & 8  & \{1\}               \\ \hline
		(24,12)     & \(S_4\)                          & 17 & (2,0,0)  & 8  & \{1\}               \\ \hline
		(24,8)      & \(C_3 \rtimes D_8\)              & 18 & (3,0,0)  & 8  & \{1\}               \\ \hline
		\rowcolor{backcolour}(27,5)      & \(C^3_3\)                        & 18 & (0,0,1)  & 7  & \(C^2_3\)           \\ \hline
		\rowcolor{backcolour}(27,5)      & \(C^3_3\)                        & 20 & (0,0,4)  & 11 & \{1\}               \\ \hline
		(27,3)      & \(C_3^2 \rtimes C_3\)                                & 18 & (0,0,1)  & 7  & \(C^2_3\)           \\ \hline
		(27,4)      &   \(C_9 \rtimes C_3\)                               & 20 & (0,0,1)  & 5  & \(C^2_3\)           \\ \hline
		(27,2)      & \(C_3 \times C_9\)               & 20 & (0,0,1)  & 5  & \(C^2_3\)           \\ \hline
		(36,10)     & \(S^2_3\)                        & 18 & (3,0,0)  & 8  & \{1\}               \\ \hline
		(36,12)     & \(S_3 \times C_6\)               & 20 & (3,0,1)  & 8  & \{1\}               \\ \hline
		(36,9)      & \(C^2_3 \rtimes C_4\)            & 18 & (1,0,0)  & 6  & \(C_2\)             \\ \hline
		(36,11)     & \(C_3 \times A_4\)               & 18 & (1,0,0)  & 6  & \(C^2_3\)           \\ \hline
		(36,7)      &   \(C_3^2 \rtimes_3 C_4\)                               & 20 & (1,2,0)  & 6  & \(C_2\)             \\ \hline
		\rowcolor{backcolour}(54,13)     & \(C_3 \times C_3 \rtimes S_3\)   & 20 & (1,0,2)  & 8  & \{1\}               \\ \hline
		\rowcolor{backcolour}(54,13)     & \(C_3 \times C_3 \rtimes S_3\)   & 18 & (1,0,1)  & 8  & \{1\}               \\ \hline
		(54,8)      & \((C_3^2 \rtimes C_3) \rtimes C_2\)             & 18 & (1,0,1)  & 8  & \{1\}               \\ \hline
		(54,12)     & \(S_3 \times C^2_3\)             & 20 & (1,0,3)  & 10 & \{1\}               \\ \hline
		(60,5)      & \(A_5\)                          & 18 & (1,0,0)  & 6  & \{1\}               \\ \hline
		(72,43)     & \(C_3 \rtimes S_4\)              & 18 & (2,0,0)  & 7  & \{1\}               \\ \hline
		(72,22)     & \(D_{12} \rtimes S_3\)           & 20 & (3,1,0)  & 7  & \{1\}               \\ \hline
		(72,40)     & \(S_3 \wr C_2\)                  & 19 & (3,0,0)  & 7  & \{1\}               \\ \hline
		(81,15)     & \(C_3^4\)                        & 20 & (0,0,10) & 23 & \{1\}               \\ \hline
		(81,12)     & \(C_3 \times (C_3^2 \rtimes C_3)\)              & 18 & (0,0,1)  & 7  & \(C^3_3\)           \\ \hline
		(81,13)     &  \(C_9 \rtimes C_3^2\)                                & 20 & (0,0,1)  & 5  & \(C^3_3\)           \\ \hline
		(81,7)      & \(C_3 \wr C_3\)                  & 20 & (0,0,2)  & 7  & \(C_3\)             \\ \hline
		(81,8)      &    \((C_3 \times C_9) \rtimes_2 C_3\)                              & 20 & (0,0,1)  & 5  & \(C_3^2 \rtimes C_3\)            \\ \hline
		(108,40)    &   \(C_3^3 \rtimes_3 C_2^2\)                               & 20 & (3,0,1)  & 8  & \{1\}               \\ \hline
		(108,38)    & \(C_3 \times S^2_3\)             & 20 & (3,0,2)  & 10 & \{1\}               \\ \hline
		\rowcolor{backcolour}(108,37)    & \(C^3_3 \rtimes C_4\)            & 20 & (1,0,1)  & 6  & \(C_2\)             \\ \hline
		\rowcolor{backcolour}(108,37)    & \(C^3_3 \rtimes C_4\)            & 20 & (1,1,0)  & 5  & \(C_2\)             \\ \hline
		\rowcolor{backcolour}(108,37)    & \(C^3_3 \rtimes C_4\)            & 19 & (1,1,0)  & 6  & \(C_2\)             \\ \hline
		(162,52)    & \(C^2_3 \times C_3 \rtimes S_3\) & 20 & (1,0,6)  & 16 & \{1\}               \\ \hline
		(162,46)    & \((C_3^2 \rtimes C_3) \rtimes_5 S_3\)           & 18 & (1,0,1)  & 8  & \{1\}               \\ \hline
		(162,10)    & \(C_3 \wr S_3\)                  & 20 & (1,0,2)  & 8  & \{1\}               \\ \hline
		(216,158)   & \(C_3^3 \rtimes_2 D_8\)          & 20 & (3,0,1)  & 8  & \{1\}               \\ \hline
		(243,65)    &   \(C_3^3 \rtimes_3 C_3^2\)                               & 18 & (0,0,1)  & 7  & \(C_3^4\)           \\ \hline
		(243,51)    & \(C_3 \times C_3 \wr C_3\)       & 20 & (0,0,4)  & 11 & \(C_3\)             \\ \hline
		(243,57)    & \((C_3\times C_9)\rtimes_2 C_3^2,\)                                 & 20 & (0,0,1)  & 5  & \(C_3 \times (C_3^2 \rtimes C_3)\) \\ \hline
		(324,167)   &    \(C_3^2 \times C_6^2\)                              & 20 & (3,0,4)  & 14 & \{1\}               \\ \hline
		(324,160)   & \(C_3^3 \rtimes A_4\)            & 20 & (1,0,1)  & 6  & \(C_3\)             \\ \hline
		(324,163)   & \(C^4_3 \rtimes_3 C_4\)          & 20 & (1,2,2)  & 10 & \(C_2\)             \\ \hline
		(360,118)   & \(A_6\)                          & 19 & (1,0,0)  & 5  & \{1\}               \\ \hline
		(405,15)    & \(C^4_3 \rtimes C_5\)            & 20 & (0,0,2)  & 7  & \(C_5\)             \\ \hline
		(486,249)   &    \((C_3^3 \rtimes_3 C_3^2) \rtimes_3 C_2 \simeq 3_+^{1+4}\rtimes_3 C_2 \)                                 & 18 & (1,0,1)  & 8  & \{1\}               \\ \hline
		(486,166)   & \(C^4_3 \rtimes_5 S_3\)          & 20 & (1,0,3)  & 10 & \{1\}               \\ \hline
		(648,704)   &   \( ((C_3^3 \rtimes C_2^2) \rtimes C_3) \rtimes C_2 \) (\textasteriskcentered)                              & 20 & (2,0,1)  & 7  & \{1\}               \\ \hline
		(648,722)   &    \(C_3^4 \rtimes D_8\)                              & 20 & (3,1,4)  & 11 & \{1\}               \\ \hline
		(729,312)   &   \( ((C_3 \times (C_9 \rtimes C_3)) \rtimes C_3) \rtimes C_3 \) (\textasteriskcentered)                                    & 20 & (0,0,2)  & 7  & \(C^2_3\)           \\ \hline
		(810,101)   &    \(C_3^4 \rtimes D_{10}\)                                 & 20 & (1,0,2)  & 8  & \{1\}               \\ \hline
		(972,877)   &    \( C_3 \times ((C_3^3 \rtimes C_2^2) \rtimes C_3) \) (\textasteriskcentered)                                                  & 20 & (1,0,2)  & 8  & \(C_3\)             \\ \hline
		(972,776)   &    \( ((C_3 \times (C_3^2 \rtimes C_3)) \rtimes C_3) \rtimes C_4 \) (\textasteriskcentered)                                             & 19 & (1,1,0)  & 6  & \(C_2\)             \\ \hline
		(1458,1229) &    \( ((C_3 \rtimes (C_3^3 \rtimes C_3)) \rtimes C_3) \rtimes C_2 \) (\textasteriskcentered)                                             & 20 & (1,0,2)  & 8  & \{1\}               \\ \hline
		(1944,3877) &    \( (C_3 \rtimes ((C_3^3 \rtimes C_2^2 ) \rtimes C_3)) \rtimes C_2 \) (\textasteriskcentered)                                             & 20 & (2,1,1)  & 8  & \{1\}               \\ \hline
		2916        &   (\textasteriskcentered)                               & 20 & (1,2,0)  & 6  & \(C_2\)             \\ \hline
		4860        &   (\textasteriskcentered)                               & 20 & (1,0,1)  & 6  & \{1\}               \\ \hline
		29160       & \(C_3^4 \rtimes A_6\)            & 20 & (1,1,0)  & 5  & \{1\}    			  \\
		\multicolumn{6}{c}{} \\ 
		\multicolumn{6}{c}{\(G \subseteq Q_8 \rtimes S_3\) } \\ 
		%&                                  &    &          &    &                     \\ \hline 
		ID          & \(G\)          & \textrm{rk}   & \(n_2\)  & \(b_2\)& \(\pi_1(Y^\reg)\)           \\ \hline
		(2, 1)      & \(C_2\)        & 8 &1 &16 & \{1\}               \\ \hline
		(4, 1)      & \(C_4\)        & 14 & 1 &10 & \(C_2\)             \\ \hline
		(4, 2)      & \(C^2_2\)      & 12 & 3 &14 & \{1\}               \\ \hline
		\rowcolor{backcolour}
		(6, 2)      & \(C_6\)        & 16 & 1 &8 & \(C_3\)             \\ \hline
		(6, 1)      & \(S_3\)        & 14 & 1 &10 & \{1\}               \\ \hline
		(8, 4)      & \(Q_8\)        & 17 & 1 &7 & \(C^2_2\)           \\ \hline
		(8, 3)      & \(D_8\)        & 15 & 3 &11 & \{1\}               \\ \hline
		(8, 1)      & \(C_8\)        & 18 & 1 &6 & \(C_4\)             \\ \hline
		(12, 4)     & \(D_{12}\)     & 16 & 3 &10 & \{1\}               \\ \hline
		(16, 8)     & \(Q_8\rtimes C_2\) & 18 & 2 &7 & \(C_2\)             \\ \hline
		(24, 3)     & \(Q_8 \rtimes C_3\) & 19 & 1 &5 & \(A_4\)             \\ \hline
		(48, 29)    & \(Q_8 \rtimes S_3\)     & 19 & 1 &6 & \{1\}               \\      
\end{longtable}
\end{center}

\subsection{Rank of invariant lattice} \label{subsec:rk}

The computation of \(\mathrm{rk} = 23 - \mathrm{rk}\,H^2(F(X))^G\) in \cref{table:term} is quite delicate. The starting point is \cite[Table 12]{HoehnMason2019}, where H{\"o}hn and Mason determine the rank of the coinvariant lattice \(H^2(X', \mathbb{Z})_{G'}\) for any finite group \(G'\) of symplectic automorphisms of hyperkähler manifolds \(X'\) of type $\textrm{K}3^{[2]}$. We recall that \(H^2(X', \mathbb{Z})_{G'}\) denotes the orthogonal complement of the invariant lattice \(H^2(X', \mathbb{Z})^{G'}\) in \(H^2(X', \mathbb{Z})\).

The groups analyzed in \cref{table:term} are all included in H{\"o}hn and Mason's list. However, the same abstract group \(G\) can appear multiple times in \cite[Table 12]{HoehnMason2019}, with different ranks for the coinvariant lattice \(H^2(X', \mathbb{Z})_G\). This variation arises because the rank depends not only on the isomorphism type of \(G\) but also on the specific action of \(G\). Notably, this occurs when \(G \simeq (3,1), (6,2), (9,2), (12,1), (18,3), (27,5), (54,13)\), or \((108,37)\); such groups are highlighted in grey in \cref{table:term}.

In order to determine the correct rank, we apply the following strategies.
\begin{enumerate} \setlength\itemsep{1em}
	\item \label{item:rkmethod_HM} For any group $G'$, \cite{HoehnMason2019} also provides the overgroups of symplectic automorphisms that contain $G'$. Therefore, the rank of $H^2(F(X), \Z)_{G}$ may be determined by the condition $G \subseteq \Aut^s(X)$. This method applies to
	\begin{itemize}
		\item $(108,37)$, and $(12,1)$ of $\Aut^s(X)\simeq(C_3^3 \rtimes_3 C_3^2)\rtimes Q_8$,
		\item $(9,2) \subset \Aut^s(X)\simeq(C_3^3 \rtimes_3 C_3^2)\rtimes Q_8$ to exclude that $\rk\,H^2(F(X), \Z)_{G}$ is $20$.
	\end{itemize}
	
	\item \label{item:rkmethod_C3} If $\Aut^s(X)\simeq(C_3^3 \rtimes_3 C_3^2)\rtimes Q_8$ or $C_3^4\rtimes A_6$, and $G\simeq \langle \diag(\om,\om,\om,1,1,1) \rangle$, a terminalization $Y$ is studied in detail in \cite[\S 3]{Kawatani2009}. In particular, $Y$ is birational to a generalised Kummer fourfolds, via a birational map which is an isomorphism in codimension 1 (see \cite[\S 3.1]{Kawatani2009}). It follows that $b_2(Y)=7$, hence $\rk\,H^2(F(X), \Z)_{G} =18$.
	
	\item \label{item:rkmethod_subgp} If $G \subseteq H$ then we have $\rk\,H^2(F(X), \Z)_{G} \leq \rk\,H^2(F(X), \Z)_{H}$. This method applies, for instance, to the subgroups 
	\begin{itemize}
		\item $(6,2) \subset (12,4)$ of $\Aut^s(X)\simeq A_7,L_2(11),A_{3,5},Q_8 \rtimes S_3$,
		\item $(12,1) \subset (24,8)$, and $(18,3) \subset (36,10)$ of $\Aut^s(X)\simeq A_{3,5}$,
		\item $(12,1) \subset (24,8)$ of $\Aut^s(X)\simeq A_{7}$,
		\item $(27,5) \subset (54,13) \subset (108,37)$, and  $(18,3) \subset (54,13)$ of $\Aut^s(X)\simeq (C_3^3 \rtimes_3 C_3^2)\rtimes Q_8$,
	\end{itemize}
	and to the overgroups
	\begin{itemize}
		\item $(3,1) \subset (6,2)$ of $\Aut^s(X)\simeq(C_3^3 \rtimes_3 C_3^2)\rtimes Q_8, C_3^4\rtimes A_6$ with $n_{32}=1$, applying also \eqref{item:rkmethod_C3},
		\item  $(3,1) \subset (9,2)$ of $\Aut^s(X)\simeq(C_3^3 \rtimes_3 C_3^2)\rtimes Q_8$, applying also \eqref{item:rkmethod_HM} and \eqref{item:rkmethod_C3}.
	\end{itemize} 
	
	\item For a smooth cubic fourfold $X \subset \PP^5$ with defining equation $f$, let $h$ be the restriction of an hyperplane class in $\PP^5$, and $\Omega = \sum_{i=0}^5 (-1)^i x_i dx_0 \wedge \ldots \wedge \widehat{dx_i} \wedge \ldots \wedge dx_n$. We recall that we have the following isomorphisms:
	\begin{align*}
		& H^2(F(X),\Z) \simeq H^4(X, \Z),\\
		& H^4(X, \C) \simeq H^{3,1}(X,\C) \oplus h^2 \oplus H^{2,2}(X,\C)_{\mathrm{prim}}  \oplus H^{1,3}(X,\C),\\
		& H^{2,0}(F(X),\C) \simeq H^{3,1}(X,\C) = \langle \mathrm{Res}\,\tfrac{\Omega}{f^2} \rangle.
	\end{align*}
	Moreover, $H^{3,1}(X,\C)$, $H^{1,3}(X,\C)$ and $h^2$ are $G$-invariant as $G$ acts symplectically and preserves $h$. In order to compute \(\rk \, H^2(F(X), \mathbb{Z})^{G}\), we reduce to study the $G$-invariant classes in $H^{2,2}(X,\C)_{\mathrm{prim}}$.
	
	\item[(4.1)] Lemma \ref{lem:rkFermat} below yields a bound on the rank of the coinvariant lattice from the study of $G$-invariant classes in $H^{2,2}(X,\C)_{\mathrm{prim}}$. Such bound completely determines $\rk \, H^2(F(X), \Z)_G$ for the following subgroups $G$ of $\Aut^s(X)\simeq C_3^4\rtimes A_6$
	\begin{itemize}
		\item $(9,2)$ with $n_{32}=2$,
		\item $(18,3)$ with $n_{32}=1$ and $\pi_1(Y^\reg)=\{1\}$, as for such group $h_3=2$,
		\item $(27,5)$ with $n_{32}=4$,
		\item $(54,13)$ with $n_{32}=2$,
		\item $(108,37)$ with $n_{32}=1$, as for such group $h_3=4$.
	\end{itemize}
	\item[(4.2)] For the following remaining subgroups $G$ of $\Aut^s(X)\simeq C_3^4\rtimes A_6$
	\begin{itemize}
		\item $(6,2)$ with $n_{32}=0$,
		\item $(9,2)$ with $n_{32}=1$,
		\item $(12,1)$,
		\item $(18,3)$ with $\pi_1(Y^\reg)=C_3$,
		\item $(27,5)$ with $n_{32}=1$,
		\item $(54,13)$ with $n_{32}=1$,
		\item $(108,37)$ with $n_{32}=0$,
	\end{itemize}
	we write explicit generators for $G$ and determine \(H^2(F(X), \mathbb{Z})^{G}\). By the proof of Lemma \ref{lem:rkFermat}, this boils down to the study of the $G$-invariant polynomials in $\langle x_ix_jx_k \,| i \neq j \neq k \rangle$. 
\end{enumerate}

\begin{lemma} \label{lem:rkFermat}
	Let $G \subseteq \Aut^s(X)\simeq C_3^4\rtimes A_6$ and $X=V(f)$ with $f=x_0^3+x_1^3+x_2^3+x_3^3+x_4^3+x_5^3$. Let $N_3$ be the number of subgroups of $G$ generated by an element of order $3$ which fixes a locus of codimension 2 in $F(X)$. 
	\begin{itemize}
		\item If $N_3 \geq 1$, then $\rk \, H^2(F(X), \Z)_G \geq 18$;
		\item if $N_3 \geq 2$, then $\rk \, H^2(F(X), \Z)_G= 20$.
	\end{itemize}
\end{lemma}
\begin{proof}
A basis for $H^{2,2}(X,\C)_{\mathrm{prim}}$ consists of $$\mathrm{Res}\,\tfrac{x_ix_jx_k\,\Omega}{f^2}$$ with $i,j$, and $k$ distinct (see for instance \cite[\S 6.1.3]{Voisina}). We reduce to study the subspace of $G$-invariant polynomials in the $\C$-vector space $W \coloneqq \langle x_ix_jx_k \,| i \neq j \neq k \rangle$.
	
Assume $N_3 \geq 1$. Up to relabelling the coordinates $x_i$ and by Proposition \ref{prop:fixed loci codim2}, we can assume that $g \coloneqq \diag(1,1,1,\om,\om,\om) \in G$. The $\langle g\rangle$-invariant polynomials in $W$ are generated by $x_0x_1x_2$ and $x_3x_4x_5$, and no other subgroup of $\Aut^s(X)$ of order $3$ and fixing a locus of codimension 2 in $F(X)$ fixes a nonzero element of $W^{\langle g \rangle}$. It follows that
$$
		\rk \, H^2(F(X), \Z)_G 
		= 23 - \rk \, H^2(F(X), \Z)^G 
		= 23 - 3 - \rk\, W^G 
		\,\begin{cases}
			\geq 20 - \rk\,W^{\langle g \rangle} = 18 & \text{if }N_3=1 \\
			= 20 & \text{if }N_3 \geq 2.
		\end{cases}
$$
\end{proof}

\begin{rem}
	Note that in \cref{table:term} identical strings $[G \subset \Aut^s(X),n_2,n_{31}, n_{32},\pi_1]$ but with different $\rk$, hence different $b_2$, witness the existence of isomorphic non-conjugate subgroups $G_i \simeq G$ of $\Aut^s(X)$ which differ in $\rk \,H^2(F(X), \mathbb{Z})^{G_i}$. For instance, this occurs when $G=(108,37) \subset \Aut^s(X)\simeq C_3^4\rtimes A_6$ and 
	\begin{align*}
	G_1 & = \langle \diag(1,1,1,\om,\om,\om), \,
	\diag(1,1,1,1,\om,\om^2)(012), \,
	\diag(1,\om^2,\om,1,1,1)(345),\\ 
	& \quad \quad (03)(1425) \rangle, 
	\\
	G_2 & = \langle \diag(1,1,1,\om,\om,\om),\, 
	\diag(1,1,1,1,\om,\om^2)(012), \,
	\diag(1,\om^2,\om,1,\om,\om^2)(345), \\ 
	& \quad \quad \diag(1,1,\om,1,1,\om^2)(03)(1425) \rangle.
	\end{align*} 
	Indeed, we have $W^{G_1}= \langle x_0x_1x_2 + x_3x_4x_5 \rangle$ hence $\rk \,H^2(F(X), \mathbb{Z})_{G_1}=19$, while $W^{G_2}= \{0\}$ hence $\rk \,H^2(F(X), \mathbb{Z})_{G_2}=20$.
\end{rem}

\subsection{New deformation classes of four-dimensional IHS varieties} 
\begin{theorem} \label{thm:b2 new} 
Let $X$ be a smooth cubic fourfold whose symplectic automorphism group $\Aut^s(X)$ is isomorphic to $L_2(11)$ or $ A_7$. Let $Y$ be a projective terminalization of $F(X)/Aut^s(X)$. Then $Y$ is an IHS variety with $b_2(Y)=4$ and simply connected regular locus, which is not deformation equivalent to any Fujiki fourfold appearing in \cite{Menet2022}, or to any terminalization of finite symplectic quotients of Hilbert squares of $K3$ surfaces, or of Kummer fourfolds by groups of induced symplectic automorphisms appearing in \cite{BertiniGrossiMauriMazzon2024}.
\end{theorem}
\begin{proof}
We recall that if two IHS varieties $Y$ and $Y'$ are deformation equivalent, then they have same second Betti number and same Fujiki constant. 

Let $G_1=L_2(11)$ and $G_2=A_7$. Let $X_i$ be a smooth cubic fourfold such that $\Aut^s(X_i)=G_i$, and $Y_i$ a projective terminalization of $X_i/G_i$, for $i=1,2$. None of the IHS fourfolds constructed in \cite{BertiniGrossiMauriMazzon2024} has second Betti number equal to $b_2(Y_i)=4$, hence none of them is deformation equivalent to $Y_i$, for $i=1,2$.

Among the Fujiki fourfolds in \cite{Menet2022}, let $Y_3\coloneqq S(G_3)_\theta^{[2]}$ be the one with $b_2(Y_3)=4$, corresponding to $G_3=A_4^2$. 
\cite[Propositions 3.19 and 3.21]{Menet2022} imply that if $Y_i$ and $Y_j$ are deformation equivalent, then the number
	\[
	\left( \frac{|G_i|}{|G_j|}\right)^{\tfrac{1}{2}}
	\]
is rational, for $i$ and $j$ in $\{1,2,3\}$. As for any $i \neq j$, this number is not rational, we conclude that $Y_1$, $Y_2$ and $Y_3$ are not deformation equivalent.
\end{proof}
Applying the same numerical criterion in the proof to the other terminalizations in Table \ref{table:term} does not allow to determine whether they define new deformation classes; see Appendix \ref{app:obstruction}.

\appendix
\section{Maximal symplectic automorphism groups} \label{appendix}
In the notation of Section \ref{sect:symplectic auto gps}, let $G$ denote a maximal symplectic automorphism group among those listed in \eqref{eq:maximalG}. By \cite[Theorems 1.2 and 1.8]{LazaZheng2022} (see also \cite[Theorem 6.12]{YangYuZhu} and \cite[Table 11]{HoehnMason2019}),
	\begin{enumerate}
		\item if $G\in \{C_3^4\rtimes A_6,(C_3^3 \rtimes_3 C_3^2)\rtimes Q_8, L_2(11),A_{3,5}\}$, then $\dim(\mathcal{M}_G)=0$ and $\mathcal{M}_G$ is irreducible;
		\item if $G\in \{A_7, M_{10}\}$, then $\dim(\mathcal{M}_G)=0$ and $\mathcal{M}_G$ has two components;
		\item if $G \simeq Q_8\rtimes S_3$, then $\dim(\mathcal{M}_G)=1$ and $\mathcal{M}_G$ is irreducible.
	\end{enumerate}
In what follows, we recall the equation $f$ of a smooth cubic fourfold $V(f)$ in each irreducible component of $\mathcal{M}_G$, and the explicit action of $G$ on $V(f)$ as subgroup of $\PGL_6$.

\begin{enumerate} \setlength\itemsep{1em}
	\item $G \simeq C_3^4\rtimes A_6$
	\begin{itemize}
		\item[-] $f=x_0^3+x_1^3+x_2^3+x_3^3+x_4^3+x_5^3$
		\item[-] action: any $\sigma \in A_6$ and $(i_1,i_2,i_3,i_4) \in C_3^4$ act by
		\begin{align*}
			\quad \quad \quad & [x_0:x_1:x_2:x_3:x_4:x_5] \mapsto \\ 
			& \quad [x_{\sigma^{-1}(0)}:\omega^{i_1}x_{\sigma^{-1}(1)}:\omega^{i_2}x_{\sigma^{-1}(2)}:\omega^{i_3}x_{\sigma^{-1}(3)}:\omega^{i_4}x_{\sigma^{-1}(4)}:\omega^{-i_1-i_2-i_3-i_4}x_{\sigma^{-1}(5)}]
		\end{align*}
		\item[-] references: \cite[Example 4.5]{Fu2015}.
		%, \cite[Second example]{Kawatani2009}.
	\end{itemize}
	
	\item $G \simeq A_7$
	\begin{itemize} 
		\item For $X^1(A_7)=V(f_1)$:
		\begin{itemize}
			\item[-] $f_1=x_0^3+x_1^3+x_2^3+x_3^3+x_4^3+x_5^3+ (-x_0-x_1-x_2-x_3-x_4-x_5)^3$
			\item[-] action: even permutations on $\{x_0,x_1,x_2,x_3,x_4,x_5, -x_0-x_1-x_2-x_3-x_4-x_5\}$
			\item[-] references: \cite[Example 4.6.2]{Mongardi2013}, \cite[Table 11]{HoehnMason2019}
		\end{itemize}
		 
		\item For $X^2(A_7)=V(f_2)$:
		\begin{itemize}
			\item[-] $f_2=x_0^3+x_1^3+x_2^3+ \tfrac{12}{5}x_0x_1x_2 + x_0x_3^2 + x_1x_4^2 + x_2x_5^2 + \tfrac{4 \sqrt{15}}{9}x_3x_4x_5$
			\item[-] action generated by
			$m_1 \coloneqq \diag(1, \om, \om^2, -1, \om, -\om^2)$ and  
			$$
			m_2 \coloneqq\begin{psmallmatrix}
				\tfrac{1}{2} & \tfrac{1}{2} & \tfrac{1}{2} & \tfrac{1}{6}\sqrt{\tfrac{5}{3}} & \tfrac{1}{6}\sqrt{\tfrac{5}{3}}&  \tfrac{1}{6}\sqrt{\tfrac{5}{3}} \\
				\tfrac{1}{2} & \tfrac{\om}{2} & \tfrac{-\xi_6}{2} & \tfrac{1}{6}\sqrt{\tfrac{5}{3}} &  \tfrac{\om}{6} \sqrt{\tfrac{5}{3}} & \tfrac{-\sqrt{5}}{36}(3i + \sqrt{3}) \\
				\tfrac{1}{2} & \tfrac{-\xi_6 }{2}& \tfrac{\om}{2} & \tfrac{1}{6}\sqrt{\tfrac{5}{3}} & \tfrac{-\sqrt{5}}{36}(3i + \sqrt{3}) &    \tfrac{\om}{6} \sqrt{\tfrac{5}{3}} \\
				\tfrac{1}{2}\sqrt{\tfrac{3}{5}} & \tfrac{1}{2}\sqrt{\tfrac{3}{5}}  &  \tfrac{1}{2}\sqrt{\tfrac{3}{5}} & \tfrac{-1}{2} & \tfrac{-1}{2} & \tfrac{-1}{2} \\
				\tfrac{1}{2}\sqrt{\tfrac{3}{5}}  & \tfrac{\om}{2} \sqrt{\tfrac{3}{5}} & \tfrac{-\xi_6}{2}\sqrt{\tfrac{3}{5}} & \tfrac{-1}{2} & \tfrac{-\om}{2} & \tfrac{\xi_6}{2} \\
				\tfrac{1}{2}\sqrt{\tfrac{3}{5}}  & \tfrac{-\xi_6}{2}\sqrt{\tfrac{3}{5}} & \tfrac{\om}{2}\sqrt{\tfrac{3}{5}} & \tfrac{-1}{2} & \tfrac{\xi_6}{2} & \tfrac{-\om}{2} 
			\end{psmallmatrix}
			$$
	
			\item references: \cite[Theorem 6.14]{YangYuZhu}
			
		\end{itemize}
	\end{itemize}
	
	\item $G \simeq (C_3^3 \rtimes_3 C_3^2)\rtimes Q_8$, IdGroup $(1944,3559)$
	\begin{itemize}
		\item[-] $f=x_0^3+x_1^3+x_2^3+x_3^3+x_4^3+x_5^3+3(i-2e^{\tfrac{\pi i}{6}}-1)(x_0x_1x_2+x_3x_4x_5)$
		\item[-] action generated by
			$$ p_1 \coloneqq \begin{psmallmatrix}
			\om & & & & & \\
			& \om & & & & \\
			& & \om & & & \\
			& & & 1 & & \\
			& & & & 1 & \\
			& & & & & 1
		\end{psmallmatrix},
		p_2 \coloneqq \begin{psmallmatrix}
			1 & & & & & \\
			& \om & & & & \\
			& & \om^2 & & & \\
			& & & 1 & & \\
			& & & & 1 & \\
			& & & & & 1
		\end{psmallmatrix},
		p_3 \coloneqq\begin{psmallmatrix}
			1 & & & & & \\
			& 1 & & & & \\
			& & 1 & & & \\
			& & & 1 & & \\
			& & & & \om & \\
			& & & & & \om^2
		\end{psmallmatrix},
		$$
		$$
		p_4 \coloneqq \begin{psmallmatrix}
			0 & 0 & 1 & & & \\
			1 & 0 & 0 & & & \\
			0 & 1 & 0 & & & \\
			& & & 1 & & \\
			& & & & 1 & \\
			& & & & & 1
		\end{psmallmatrix},
		p_5 \coloneqq \begin{psmallmatrix}
			1 & & & & & \\
			& 1 & & & & \\
			& & 1 & & & \\
			& & & 0 & 0 & 1 \\
			& & & 1 & 0 & 0 \\
			& & & 0 & 1 & 0
		\end{psmallmatrix},
		p_6 \coloneqq \begin{psmallmatrix}
			0& 1& & & & \\
			1& 0 & & & & \\
			& & 1& & & \\
			& & &0 &1 & \\
			& & &1&0 & \\
			& & & & &1
		\end{psmallmatrix},
		$$
		$$
		p_7 \coloneqq \begin{psmallmatrix}
			& & &1 & & \\
			& & & &1 & \\
			& & & & &1 \\
			0& 1& 0& & & \\
			1& 0& 0& & & \\
			0& 0& 1& & &
		\end{psmallmatrix},
		p_8 \coloneqq \tfrac{1}{\sqrt{3}} \begin{psmallmatrix}
			\om & \om^2 & 1 & & &  \\
			1 & 1 & 1 & & & \\
			\om^2 & \om & 1 & & & \\
			& & & \om^2 & \om & 1 \\
			& & & \om^2 & \om^2 & \om^2 \\
			& & & \om^2 & 1 & \om 
		\end{psmallmatrix};
		$$ 
		\item[-] references: \cite[Example 4.2.6]{Mongardi2013}, \cite[\S 8, p.244]{HoehnMason2019}.
	\end{itemize}
	
	\item $G \simeq M_{10}$, IdGroup $(720,765)$
	\begin{itemize} 
		\item For $X^1(M_{10})=V(f_1)$:
		\begin{itemize}
			\item[-] $f_1=x_0^3+x_1^3+x_2^3+x_3^3+x_4^3+x_5^3+ \tfrac{1}{5}(-3\xi^7-3\xi^5+3\xi^4-3\xi^3+6\xi-3) \cdot \big[x_0x_1x_2 + x_0x_1x_3+ (\xi^4-1)x_0x_1x_4 + x_0x_1x_5 + (\xi^4-1)x_0x_2x_3 + x_0x_2x_4 + x_0x_2x_5 + (\xi^4-1)x_0x_3x_4 - \xi^4 x_0x_3x_5 - \xi^4 x_0x_4x_5 + (\xi^4-1)x_1x_2x_3 + (\xi^4-1)x_1x_2x_4 - \xi^4 x_1x_2x_5 + x_1x_3x_4 + x_1x_3x_5 - \xi^4 x_1x_4x_5 + x_2x_3x_4 - \xi^4x_2x_3x_5 + x_3x_4x_5 + x_3x_4x_5\big]$  where $\xi=e^{\tfrac{\pi i}{12}}$
			\item[-] action generated by 
			$$g_1 \coloneqq \begin{psmallmatrix}
				1 & & & & &  \\
				& 0 & 1 & & & \\
				& 1 & 0 & & & \\
				& & & 0 & 1 & \\
				& & & 1 & 0 & \\
				& & & & & 1
			\end{psmallmatrix},
			g_2 \coloneqq \begin{psmallmatrix}
				0 & 1 & 0 & 0 & & \\
				0 & 0 & \om & 0 & & \\
				0 & 0 & 0 & 1 & & \\
				\om^2 & 0 & 0 & 0 & & \\
				& & & & 0 & 1 \\
				& & & & 1 & 0
			\end{psmallmatrix},
			\tfrac{1}{\sqrt{6}} \begin{psmallmatrix}
				1 & \om & \om^2 & \om & 1 & \om \\
				\om^2 & 1 & 1 & \om & \om & \om \\
				\om & 1 & \om^2 & \om & \om^2 & \om^2 \\
				\om^2 & \om^2 & \om & \om & 1 & \om^2 \\
				1 & \om^2 & 1 & \om & \om^2 & 1 \\
				\om^2 & \om^2 & \om^2 & \om^2 & \om^2 & \om 
			\end{psmallmatrix};
			$$
			\item[-] references: \cite[\S 8, p.244]{HoehnMason2019}
		\end{itemize}
		
		\item For $X^2(M_{10})=V(f_2)$:
		\begin{itemize}
			\item[-] explicit equation of $f_2$ provided in \cite[Theorem 6.15]{YangYuZhu}  
			\item[-] action generated by $\diag(1,-1,i,-i,\xi_8^7,-\xi_8)$ and the matrix in the supplementary file \texttt{Theorem6.15.txt} of \cite{YangYuZhu}
			\item references: \cite[Theorem 6.15]{YangYuZhu}
		\end{itemize}
	\end{itemize}
	
	\item $G \simeq L_2(11)=\PSL_2(\mathbb{F}_{11})$, IdGroup $(660,13)$
	\begin{itemize}
		\item[-] $f=x_0^3+x_1^2x_5+x_2^2x_4+x_3^2x_2+x_4^2x_1+x_5^2x_3$
		\item[-] action generated by
		$$
		h_1\coloneqq\begin{psmallmatrix}
			1 & 0 & 0 & 0 & 0 & 0 \\
			0 & 0 & 1 & 0 & 0 & 0 \\
			0 & 1 & 0 & 0 & 0 & 0 \\
			0 & 0 & 0 & 0 & 0 & 1 \\
			0 & 1 & -1 & c & 1 & c \\
			0 & 0 & 0 & 1 & 0 & 0
		\end{psmallmatrix},
		h_2 \coloneqq\begin{psmallmatrix}
			1 & 0 & 0 & 0 & 0 & 0 \\
			0 & 0 & 0 & 0 & 1 & 0 \\
			0 & 0 & 0 & 1 & 0 & 0  \\
			0 & 0 & -1 & -1 & 0 & 0 \\
			0 & -C & 0 & 0 & -1 & -c-2C  \\
			0 & 1 & 0 & 0 & -1 & 1 
		\end{psmallmatrix},
		$$ where $c=\zeta + \zeta^3+\zeta^4+\zeta^5+\zeta^9$, $C=-1-c$, and $\zeta=e^{\tfrac{2 \pi i}{11}}$; note that $(h_1 h_2)^{11}=1$
		\item[-] references: \cite{Adler1978}, \cite[Theorem 1.3]{Mongardi2013a}, ATLAS.
	\end{itemize}
	
	\item $G \simeq A_{3,5}$, IdGroup $(360,120)$
	\begin{itemize}
		\item[-] $f=x_0^3+x_1^3+(-x_0-x_1)^3 + x_2^3+x_3^3+x_4^3+x_5^3+ (-x_2-x_3-x_4-x_5)^3$
		\item[-] action: given the permutation groups $S_3$ on $\{x_0,x_1,-x_0-x_1\}$, and $S_5$ on $\{x_2,x_3,x_4,x_5,-x_2-x_3-x_4-x_5\}$, the action of $G$ is given by the even permutations $(S_3 \times S_5)\cap A_8$
		\item[-] reference: \cite[Theorem 1.8]{LazaZheng2022}.
	\end{itemize}
	
	\item $G \simeq Q_8 \rtimes S_3$, IdGroup $(48,29)$
	\begin{itemize}
		\item[-] $\mathcal{M}_G$ is one-dimensional, irreducible and contains $V(f)$ where $f=x_0^3+x_0(x_1^2+x_2^2+x_3^2)+x_1x_2x_3+ix_2(x_4^2+x_5^2)+2x_1x_4x_5+x_3(x_4^2-x_5^2)$
		\item[-] action: generators of $S_3$ 
		$$ n_1 \coloneqq \begin{psmallmatrix}
			1 & & & & &  \\
			& 0 & 1& 0& & \\
			& 1 & 0& 0& & \\
			& 0 & 0& 1 & & \\
			& & & & \tfrac{1}{\sqrt{2}} & \tfrac{-1}{\sqrt{2}} \\
			& & & & \tfrac{i}{\sqrt{2}} & \tfrac{-1}{\sqrt{2}}
		\end{psmallmatrix},
		n_2 \coloneqq \begin{psmallmatrix}
			1 & & & & &  \\
			& 1 & & & & \\
			& & 0 & 1& & \\
			& & 1 & 0& & \\
			& & & & 0 & \xi_8^3 \\
			& & & & \xi_8^5 & 0
		\end{psmallmatrix},
		$$
		and generators of $Q_8$
		$$ \begin{psmallmatrix}
			1 & & & & &  \\
			& -1 & & & & \\
			& & 1 & & & \\
			& & & -1 & & \\
			& & & & 0 &  1\\
			& & & & -1& 0
		\end{psmallmatrix},\diag(1,1,-1,-1,i,-i);
		$$
		\item[-] references: \cite[\S 7.6]{Koike2024} 
	\end{itemize}
	
\end{enumerate} 

\section{Supporting GAP code}\label{app:Gap}

\subsection{Elements of order 3 with normal form $\diag(1,1,1,\omega,\omega,\omega)$}\label{app:L3}
Let $G$ be a group acting symplectically on a smooth cubic fourfold. By Proposition \ref{prop:fixed loci codim2}, we have to identify the elements of $G$ of order 3 and having normal form $\diag(1,1,1,\omega,\omega,\omega)$. This last condition is not purely group-theoretic, in the sense that it depends on the embedding of $G$ in $\mathrm{PGL}_6$, hence on the action of $G$ on the cubic fourfold. 

The GAP code below produces a list \texttt{L3} that contains all elements of $G$ of order 3 and having normal form $\diag(1,1,1,\omega,\omega,\omega)$, for the actions of the groups $G$ in Appendix \ref{appendix}.
\bigskip

\noindent \textbf{For the groups $A_{3,5}, L_2(11), M_{10}, A_7, Q_8 \rtimes S_3$}

\begin{lstlisting}[language=GAP]
#Define G=A_{3,5} by the generators l, t, s in PGL_6. As in G there are precisely 3 conjugacy classes of subgroups of order 3, construct a subgroup for each class, compute the characteristic polynomial of a generator, and compare it with the characteristic polynomial of diag(1,1,1,w,w,w). Set L3 as the empty list.
w:=E(3);
l:=[[0,-1,0,0,0,0],[1,-1,0,0,0,0],[0,0,1,0,0,0],
[0,0,0,1,0,0],[0,0,0,0,1,0],[0,0,0,0,0,1]];;
t:=[[0,1,0,0,0,0],[1,0,0,0,0,0],[0,0,0,1,0,0],
[0,0,1,0,0,0],[0,0,0,0,1,0],[0,0,0,0,0,1]];;
s:=[[1,0,0,0,0,0],[0,1,0,0,0,0],[0,0,0,0,0,-1],
[0,0,1,0,0,-1],[0,0,0,1,0,-1],[0,0,0,0,1,-1]];;
G:=GroupByGenerators([l,t,s]);;
H1:=Subgroup(G,[l]);;
H2:=Subgroup(G,[t*s*t*s^-1]);;
H3:=Subgroup(G,[l*t*s*t*s^-1]);;
IsConjugate(G,H1,H2);
IsConjugate(G,H1,H3);
IsConjugate(G,H2,H3);
CharacteristicPolynomial(l);
CharacteristicPolynomial(t*s*t*s^-1);
CharacteristicPolynomial(l*t*s*t*s^-1);
CharacteristicPolynomial(DiagonalMat([1,1,1,w,w,w]));
L3:=[];;

#In G=L_2(11) there is precisely 1 conjugacy class of subgroups of order 3, and a subgroup for this class is generated by h2, one of the generators of the action of G in PGL_6. Construct h2, compute its characteristic polynomial, and compare it with the characteristic polynomial of diag(1,1,1,w,w,w). Set L3 as the empty list.
z:=E(11);
c:=z+z^3+z^4+z^5+z^9;
C:=-1-c;
h2:=[[1,0,0,0,0,0],[0,0,0,0,1,0],[0,0,0,1,0,0],
[0,0,-1,-1,0,0],[0,-C,0,0,-1,-c-2*C],[0,1,0,0,-1,1]];;
CharacteristicPolynomial(h2);
CharacteristicPolynomial(DiagonalMat([1,1,1,w,w,w]));
L3:=[];;

#In G=M_10 there is precisely 1 conjugacy class of subgroups of order 3, and a subgroup for this class is generated by the element g, expressed in terms of the generators of the action of G in PGL_6 on X^1(G). Construct g, compute its characteristic polynomial, and compare it with the characteristic polynomial of diag(1,1,1,w,w,w). Set L3 as the empty list.
w:=E(3);
g1:=[[1,0,0,0,0,0],[0,0,1,0,0,0],[0,1,0,0,0,0],
[0,0,0,0,1,0],[0,0,0,1,0,0],[0,0,0,0,0,1]];;
g2:=[[0,1,0,0,0,0],[0,0,w,0,0,0],[0,0,0,1,0,0],
[w^2,0,0,0,0,0],[0,0,0,0,0,1],[0,0,0,0,1,0]];;
g:=g1*g2*(g1*g2^-1)^2*g1*g2^2*g1*g2^-1;
CharacteristicPolynomial(g);
CharacteristicPolynomial(DiagonalMat([1,1,1,w,w,w]));
L3:=[];;

#Define the actions of G=M_10 on X^2(G), of G=A_7 on X^1(G) and X^2(G), and of G=Q_8 \rtimes S_3, by constructing their generators in PGL_6. In ord3 list all elements of G of order 3, and in listpoly all the corresponding characteristic polynomials. Compare listpoly with the characteristic polynomial of diag(1,1,1,w,w,w). Set L3 as the empty list. 

#For G=M_10 acting on X^2(G):
k1:=[ [ 1/4*E(3)^2, 17/44*E(24)-1/22*E(24)^8+5/22*E(24)^11-4/11*E(24)^14-5/44*E(24)^16+3/22*E(24)^17+23/44*E(24)^19-6/11*E(24)^22, 3/22*E(24)+1/44*E(24)^8+4/11*E(24)^11+5/44*E(24)^14-1/44*E(24)^16-1/22*E(24)^17+13/44*E(24)^19-3/22*E(24)^22,       -5/44*E(24)-27/44*E(24)^8+1/11*E(24)^11+7/44*E(24)^14-9/44*E(24)^16-15/22*E(24)^17+1/2*E(24)^19-5/11*E(24)^22, 1/44*E(24)+1/11*E(24)^8-2/11*E(24)^11+1/22*E(24)^14-3/44*E(24)^16-1/22*E(24)^17+1/44*E(24)^19-1/4*E(24)^22, -3/44*E(24)+5/11*E(24)^8-1/2*E(24)^11-1/22*E(24)^14+15/44*E(24)^16+4/11*E(24)^17-7/44*E(24)^19+5/44*E(24)^22 ], 
[17/44*E(24)+1/22*E(24)^8+5/22*E(24)^11+4/11*E(24)^14+5/44*E(24)^16+3/22*E(24)^17+23/44*E(24)^19+6/11*E(24)^22,-1/4*E(3)^2, -5/44*E(24)+27/44*E(24)^8+1/11*E(24)^11-7/44*E(24)^14+9/44*E(24)^16-15/22*E(24)^17+1/2*E(24)^19+5/11*E(24)^22, 3/22*E(24)-1/44*E(24)^8+4/11*E(24)^11-5/44*E(24)^14+1/44*E(24)^16-1/22*E(24)^17+13/44*E(24)^19+3/22*E(24)^22,       1/44*E(24)-1/11*E(24)^8-2/11*E(24)^11-1/22*E(24)^14+3/44*E(24)^16-1/22*E(24)^17+1/44*E(24)^19+1/4*E(24)^22, -3/44*E(24)-5/11*E(24)^8-1/2*E(24)^11+1/22*E(24)^14-15/44*E(24)^16+4/11*E(24)^17-7/44*E(24)^19-5/44*E(24)^22 ], 
[ 11/292*E(24)-35/292*E(24)^8-10/73*E(24)^11+69/292*E(24)^14-23/292*E(24)^16-22/73*E(24)^17-17/73*E(24)^19+55/146*E(24)^22, -155/146*E(24)-321/292*E(24)^8+119/146*E(24)^11+53/292*E(24)^14-115/292*E(24)^16-147/146*E(24)^17-121/292*E(24)^19+129/146*E(24)^22, -1/4*E(12)^11, -55/292*E(24)-11/73*E(24)^8-23/73*E(24)^11+83/146*E(24)^14+47/73*E(24)^16-36/73*E(24)^17-317/292*E(24)^19+253/292*E(24)^22,       -91/292*E(24)-51/146*E(24)^8+23/73*E(24)^11-5/73*E(24)^14-115/292*E(24)^16-1/146*E(24)^17+25/292*E(24)^19+39/292*E(24)^22, 93/292*E(24)+81/146*E(24)^8-43/146*E(24)^11+38/73*E(24)^14+217/292*E(24)^16+33/73*E(24)^17+29/292*E(24)^19-19/292*E(24)^22], 
[ -155/146*E(24)+321/292*E(24)^8+119/146*E(24)^11-53/292*E(24)^14+115/292*E(24)^16-147/146*E(24)^17-121/292*E(24)^19-129/146*E(24)^22, 11/292*E(24)+35/292*E(24)^8-10/73*E(24)^11-69/292*E(24)^14+23/292*E(24)^16         -22/73*E(24)^17-17/73*E(24)^19-55/146*E(24)^22, -55/292*E(24)+11/73*E(24)^8-23/73*E(24)^11-83/146*E(24)^14-47/73*E(24)^16-36/73*E(24)^17-317/292*E(24)^19-253/292*E(24)^22, 1/4*E(12)^11, -91/292*E(24)+51/146*E(24)^8+23/73*E(24)^11+5/73*E(24)^14 +115/292*E(24)^16-1/146*E(24)^17+25/292*E(24)^19-39/292*E(24)^22, 93/292*E(24)-81/146*E(24)^8-43/146*E(24)^11-38/73*E(24)^14-217/292*E(24)^16+33/73*E(24)^17+29/292*E(24)^19+19/292*E(24)^22 ],   
[ 3/8*E(24)-1/2*E(24)^8+1/4*E(24)^14-3/8*E(24)^16+1/4*E(24)^17-1/8*E(24)^19+1/8*E(24)^22, 3/8*E(24)+1/2*E(24)^8-1/4*E(24)^14+3/8*E(24)^16+1/4*E(24)^17-1/8*E(24)^19-1/8*E(24)^22,       1/8*E(24)+1/4*E(24)^11+1/4*E(24)^14-3/8*E(24)^16-1/8*E(24)^19+1/8*E(24)^22, 1/8*E(24)+1/4*E(24)^11-1/4*E(24)^14+3/8*E(24)^16-1/8*E(24)^19-1/8*E(24)^22, 0, E(24)-1/2*E(24)^11+1/2*E(24)^17-1/2*E(24)^19 ],   
[ 7/104*E(24)+3/26*E(24)^8-5/52*E(24)^11-9/52*E(24)^14+3/104*E(24)^16-3/13*E(24)^17-9/104*E(24)^19-11/104*E(24)^22, 7/104*E(24)-3/26*E(24)^8-5/52*E(24)^11+9/52*E(24)^14-3/104*E(24)^16-3/13*E(24)^17-9/104*E(24)^19+11/104*E(24)^22,       15/104*E(24)+2/13*E(24)^8+3/26*E(24)^11+1/52*E(24)^14+17/104*E(24)^16+17/52*E(24)^17-23/104*E(24)^19+7/104*E(24)^22, 15/104*E(24)-2/13*E(24)^8+3/26*E(24)^11-1/52*E(24)^14-17/104*E(24)^16+17/52*E(24)^17-23/104*E(24)^19-7/104*E(24)^22, -3/26*E(24)-5/26*E(24)^11+1/26*E(24)^17+1/13*E(24)^19, 0 ] ];;
k2:=DiagonalMat([1,-1,E(4),-E(4),E(8)^7,-E(8)]);;
G:=GroupByGenerators([k1,k2]);;

#For G=A_7 acting on X^1(G):
i1:=[[0,1,0,0,0,0],[0,0,1,0,0,0],[1,0,0,0,0,0],
[0,0,0,1,0,0],[0,0,0,0,1,0],[0,0,0,0,0,1]];;
i2:=[[0,0,0,0,1,0],[0,1,0,0,0,0],[0,0,0,1,0,0],
[0,0,0,0,0,1],[-1,-1,-1,-1,-1,-1],[0,0,1,0,0,0]];;
G:=GroupByGenerators([i1,i2]);;

#For G=A_7 acting on X^2(G):
m1:=[[1,0,0,0,0,0],[0,E(3),0,0,0,0],[0,0,E(3)^2,0,0,0],
[0,0,0,-1,0,0],[0,0,0,0,E(3),0],[0,0,0,0,0,-E(3)^2]];;
m2:=[[1/2,1/2,1/2,Sqrt(5/3)/6,Sqrt(5/3)/6,    Sqrt(5/3)/6],
[1/2, (1/2)*E(3),-(1/2)*E(6),Sqrt(5/3)/6,1/6 *E(3)*Sqrt(5/3),-(1/36)*Sqrt(5)*(3*E(4)+Sqrt(3))], 
[1/2,-(1/2)*E(6),(1/2)*E(3),Sqrt(5/3)/6, -(1/36)*Sqrt(5)*(3*E(4)+Sqrt(3)),1/6 *E(3)*Sqrt(5/3)], 
[Sqrt(3/5)/2, Sqrt(3/5)/2, Sqrt(3/5)/2, -(1/2),-(1/2),-(1/2)], 
[Sqrt(3/5)/2,1/2 *E(3)*Sqrt(3/5),-(1/2)*E(6)* Sqrt(3/5),-(1/2),-(1/2)*E(3),1/2*E(6)], 
[Sqrt(3/5)/2,-(1/2)*E(6)*Sqrt(3/5),1/2*E(3)*Sqrt(3/5),-(1/2),1/2*E(6),-(1/2)*E(3)]];;
G:=GroupByGenerators([m1,m2]);;

#For G=Q_8 \rtimes S_3:
n1:=[[1,0,0,0,0,0],[0,0,1,0,0,0],[0,1,0,0,0,0],[0,0,0,1,0,0],
[0,0,0,0,Sqrt(2)^-1, E(4)/Sqrt(2)],[0,0,0,0,-E(4)/Sqrt(2), -1/Sqrt(2)]];;
n2:=[[1,0,0,0,0,0],[0,1,0,0,0,0],[0,0,0,1,0,0],
[0,0,1,0,0,0],[0,0,0,0,0,E(8)^5],[0,0,0,0,E(8)^3,0]];;
n3:=[[1,0,0,0,0,0],[0,-1,0,0,0,0],[0,0,1,0,0,0],
[0,0,0,-1,0,0],[0,0,0,0,0,-1],[0,0,0,0,1,0]];;
n4:=[[1,0,0,0,0,0],[0,1,0,0,0,0],[0,0,-1,0,0,0],
[0,0,0,-1,0,0],[0,0,0,0,E(4),0],[0,0,0,0,0,-E(4)]];;
G:=GroupByGenerators([n1,n2,n3,n4]);;

#Construct ord3, listpoly and L3:
G3:=[];;
for i in Elements(G) do
    if IsDiagonalMatrix(i^3)=true then
          Add(G3,[i,i^3]);
    fi;
od;
mid:= [ [ 1, 0, 0, 0, 0, 0 ], [ 0, 1, 0, 0, 0, 0 ], [ 0, 0, 1, 0, 0, 0 ], [0, 0, 0, 1, 0, 0 ], [ 0, 0, 0, 0, 1, 0 ],[ 0, 0, 0, 0, 0, 1 ] ];;
ord3:=[];;
for i in G3 do
    if i[2]=mid then
          Add(ord3, i[1]);
    fi;
od;
list3:=[];;
for i in ord3 do
        Add(list3,[i,CharacteristicPolynomial(i)]);
od;
listpoly:=[];;
for i in list3 do
    a:=0;
    for j in listpoly do
         if i[2]=j then
              a:=1;
         fi;
    od;
    if a=0 then
        Add(listpoly,i[2]);
    fi;
od;
listpoly;
CharacteristicPolynomial(DiagonalMat([1,1,1,E(3),E(3),E(3)]));
L3:=[];;
\end{lstlisting}

\noindent \textbf{For the group $(C_3^3 \rtimes_3 C_3^2)\rtimes Q_8$}

\begin{lstlisting}[language=GAP]
#Define the group G=C_3^3 \rtimes_3 C_3^2)\rtimes Q_8 by the generators p_i's. In listpoly list all the characteristic polynomials corresponding to the elements of G of order 3. Compare listpoly with the characteristic polynomial of diag(1,1,1,w,w,w). Set L3 as the list containing a generator of the unique subgroup of G of order 3 and having normal form diag(1,1,1,w,w,w). 
w:=E(3);;
t:=ER(3);
p1:=DiagonalMat([w,w,w,1,1,1]);;
p12:=DiagonalMat([1,1,1,w,w,w]);; #the square of p1 in PGL_6
p2:=DiagonalMat([1,w,w^2,1,1,1]);;
p3:=DiagonalMat([1,1,1,1,w,w^2]);;
p4:=[[0,1,0,0,0,0],[0,0,1,0,0,0],[1,0,0,0,0,0],
[0,0,0,1,0,0],[0,0,0,0,1,0],[0,0,0,0,0,1]];;
p5:=[[1,0,0,0,0,0],[0,1,0,0,0,0],[0,0,1,0,0,0],
[0,0,0,0,1,0],[0,0,0,0,0,1],[0,0,0,1,0,0]];;
p6:=[[0,1,0,0,0,0],[1,0,0,0,0,0],[0,0,1,0,0,0],
[0,0,0,0,1,0],[0,0,0,1,0,0],[0,0,0,0,0,1]];;
p7:=[[0,0,0,0,1,0],[0,0,0,1,0,0],[0,0,0,0,0,1],
[1,0,0,0,0,0],[0,1,0,0,0,0],[0,0,1,0,0,0]];;
p8:=[[w/t,w^2/t,1/t,0,0,0],[1/t,1/t,1/t,0,0,0],[w^2/t,w/t,1/t,0,0,0],
[0,0,0,w^2/t,w/t,1/t],[0,0,0,w^2/t,w^2/t,w^2/t],[0,0,0,w^2/t,1/t,w/t]];;
V:=GroupByGenerators([p1,p12,p2,p3,p4,p5,p6,p7,p8]);
dV:=Center(V);;
G:=V/dV;; #Projection of V from GL_6 to PGL_6
epiVG:=GQuotients(V,G);; #Compute the surjective morphisms from V to G
f:=epiVG[1];;
G3:=[];;
for i in Elements(V) do
    if IsDiagonalMatrix(i^3)=true then
          Add(G3,[i,CharacteristicPolynomial(i)]);
    fi;
od;
listpoly:=[];;
for j in G3 do
    if j[2]=CharacteristicPolynomial(DiagonalMat([1,1,1,E(3),E(3),E(3)])) then
          Add(listpoly,j[1]);
    fi;
od;
x:=Image(f,listpoly[1]);;
L3:=[x];;
\end{lstlisting}

\noindent \textbf{For the group $C_3^4\rtimes A_6$}

\begin{lstlisting}[language=GAP]
#Define the group G=C_3^4\rtimes A_6 as semidirect product, and the list L3 applying Lemma 2.5.
A:=SmallGroup(360,118);;
genA:=GeneratorsOfGroup(A);; #output (1,2,3,4,5) and (2,3,4,5,6)
C3:=SmallGroup(3,1);;
C:=DirectProduct(C3,C3,C3,C3);;
genC:=GeneratorsOfGroup(C);;
a:=genC[1];;
b:=genC[2];;
c:=genC[3];;
d:=genC[4];;
Aut:=AutomorphismGroup(C);;
f1:=function(h)
> local p1, p2, p3, p4, e1, e2, e3, e4, x, y, z, w;
> p1:=Projection(C,1);
> p2:=Projection(C,2);
> p3:=Projection(C,3);
> p4:=Projection(C,4);
> e1:=Embedding(C,1);
> e2:=Embedding(C,2);
> e3:=Embedding(C,3);
> e4:=Embedding(C,4);
> x:=h^p1;
> y:=h^p2;
> z:=h^p3;
> w:=h^p4;
> return (w^-1)^e1*x^e2*(w^-1)^e2*y^e3*(w^-1)^e3*z^e4*(w^-1)^e4;
> end;
f1sd:=GroupHomomorphismByFunction(C,C,f1);;
f2:=function(h)
> local p1, p2, p3, p4, e1, e2, e3, e4, x, y, z, w;
> p1:=Projection(C,1);
> p2:=Projection(C,2);
> p3:=Projection(C,3);
> p4:=Projection(C,4);
> e1:=Embedding(C,1);
> e2:=Embedding(C,2);
> e3:=Embedding(C,3);
> e4:=Embedding(C,4);
> x:=h^p1;
> y:=h^p2;
> z:=h^p3;
> w:=h^p4;
> return (x^-1)^e1*(y^-1)^e1*(z^-1)^e1*(w^-1)^e1*x^e2*y^e3*z^e4;
> end;
f2sd:=GroupHomomorphismByFunction(C,C,f2);;
alpha:=GroupHomomorphismByImages(A,Aut,genA,[f1sd,f2sd]);;
G:=SemidirectProduct(A,alpha,C);;
l3:=[a*b, a*c, a*d, a*b*c, a*b*d, a*c*d, b*c, b*d, c*d, b*c*d];;
L3:=[];;
for i in l3 do
    Add(L3, Embedding(G,2)(i));
od;
\end{lstlisting}

\subsection{Codimension 2 fixed loci and the fundamental group}\label{app:invariants}

Let $G$ be one of the groups in Appendix \ref{appendix}, which acts symplectically on a smooth cubic fourfold $X$. Let \texttt{L3} be the set of subgroups of $G$ generated by an element of order $3$ with normal form $\diag(1,1,1,\omega, \omega, \omega)$, hence an element of order $3$ which fixes a codimension 2 locus in $F(X)$; see Proposition \ref{prop:fixed loci codim2}.

For each subgroup $H$ of $G$, the GAP code below computes the number of components of the singular locus of $F(X)/H$ of codimension 2 -- and whether their blowup extracts one or two exceptional divisors -- and the fundamental group of the smooth locus of a terminalization of $F(X)/H$. More precisely, the code first lists in \texttt{cc} all subgroups of $G$, up to conjugation. Then, for each subgroup $H$ in \texttt{cc}, it computes 
    \begin{itemize}
        \item the group ID or the size of $H$;
        \item \texttt{fix}, the subgroups of $H$ of order 2 or contained in \texttt{L3};
        \item \texttt{hcong2}, the subgroups of $H$ of order 2, up to conjugation; 
        \item \texttt{h3}, the subgroups of $H$ contained in \texttt{L3};
        \item \texttt{hcong3}, the conjugacy classes of subgroups of $H$ contained in \texttt{L3};
        \item \texttt{hcong31}, the conjugacy classes of subgroups $\langle g \rangle$ of $H$ contained in \texttt{L3} and such that $N_H(g)\setminus C_H(g)$ contains elements of even order;
        \item \texttt{hcong32}, the conjugacy classes of subgroups $\langle g \rangle$ of $H$ contained in \texttt{L3} and such that $N_H(g)\setminus C_H(g)$ contains no elements of even order;
        \item \texttt{N}, the subgroup of $H$ generated by the elements in \texttt{fix};
        \item \texttt{pai}, the quotient of $H$ by $N$.
    \end{itemize}
The size of \texttt{hcong2}, \texttt{h3}, \texttt{hcong3}, \texttt{hcong31}, and \texttt{hcong32} gives, respectively, the invariants $n_2$, $N_3$, $n_3$, and $n_{31}$, $n_{32}$; see Theorem \ref{thm:b2} and Proposition \ref{prop:b2_grouptheoretic}. The quotient \texttt{pai} is the fundamental group of the smooth locus of a terminalization of $F(X)/H$, by \cite[Proposition 8.1]{BertiniGrossiMauriMazzon2024}

Finally, the list \texttt{output} collects a string
$$
[\text{IdGroup}(H), n_2, N_3, n_3, [n_{31}, n_{32}], \text{IdGroup}(\texttt{pai})]
$$
for each subgroup $H$ such that $n_2+n_3 > 0$, hence such that the singular locus of $F(X)/H$ has some components of codimension 2 (or equivalently $F(X)/H$ is not terminal).
\bigskip

\noindent \textbf{Preamble for the groups $A_{3,5}, L_2(11), M_{10}, A_7, Q_8 \rtimes S_3$}

\begin{lstlisting}[language=GAP]
G:=SmallGroup(360,120);; #Define the group G=A_{3,5}
G:=SmallGroup(660,13);; #Define the group G=L_2(11)
G:=SmallGroup(720,765);; #Define the group G=M_{10}
G:=AlternatingGroup(7);; #Define the group G=A_7
G:=SmallGroup(48,29);; #Define the group G=Q_8 \rtimes S_3
L3:=[];;
\end{lstlisting}

\noindent \textbf{Preamble for the groups $(C_3^3 \rtimes_3 C_3^2)\rtimes Q_8, C_3^4\rtimes A_6$}

\begin{lstlisting}[language=GAP]
#For G=(C_3^3 \rtimes_3 C_3^2)\rtimes Q_8 and G=C_3^4\rtimes A_6, define G and construct L3 as in Appendix B.1
\end{lstlisting}

\noindent \textbf{For any group $G$}

\begin{lstlisting}[language=GAP]
#Create a list cc with all subgroups H of G, up to conjugation:
cc:= List(ConjugacyClassesSubgroups(G), Representative);;
s:=Size(cc);;

#Create a list outputall that, for each subgroup H in cc, gives an entry with [#H in cc, IdGroup(H) or Size(H),n_2,N_3,n_3,[n_31,n_32],IdGroup(pai)]: 
outputall:=[];;
for i in [1..s] do
    H:=cc[i];
    h3:=[];
    hcong3:=[];
    hcong31:=[];
    hcong32:=[];
    hcong2:=[];
    fix:=[];
    for j in Elements(H) do
        if j in L3 then
            Add(h3,j);
            Add(fix,j);
            x:=0;
            for l in hcong3 do
                for g in Elements(H) do
                    if j=g*l*g^-1 then
                            x:=1;
                    elif j=g*l^2*g^-1 then
                            x:=1;
                    fi;
                od;
            od;
            if x=0 then
                Add(hcong3,j);
                Nj:=Normalizer(H,j);
                Cj:=Centralizer(H,j);
                zj:=0;
                for k in Elements(Nj) do
                    if (k in Cj)=false then
                        if IsEvenInt(Order(k)) then 
                            zj:=1;
                        fi;
                    fi;
                od;
                if zj=0 then
                    Add(hcong32,j);
                elif zj=1 then
                    Add(hcong31,j);
                fi;
            fi;
        elif IsOne(j)=false then
            if IsOne(j^2)=true then
             	Add(fix,j);
             	y:=0;
             	for l in hcong2 do
                   for g in Elements(H) do
                       if j=g*l*g^-1 then
                            y:=1;
                       fi;
                   od;
              	od;
              	if y=0 then
                    Add(hcong2,j);
              	fi;
	        fi;
        fi;
    od;
    N:=Subgroup(H,fix);
    pai:=H/N;
    if Size(H)<2520 then
        Add(outputall,[i,IdGroup(H),Size(hcong2),
        Size(h3),Size(hcong3),[Size(hcong31),Size(hcong32)],IdGroup(pai)]);
    else
        Add(outputall,[i,[Size(H),0],Size(hcong2),
        Size(h3),Size(hcong3),[Size(hcong31),Size(hcong32)],IdGroup(pai)]);
    fi;
od;

#Create a list output that extracts from the list outputall the subgroups H which contain elements fixing a locus of codimension 2:
output:=[];;
for i in outputall do
    if i[3]+i[5]>0 then
         Add(output,i);
    fi;
od;
\end{lstlisting}

\subsection{Second Betti number}
Let $G$ be one of the groups in Appendix \ref{appendix}, which acts symplectically on a smooth cubic fourfold $X$. Let $H$ be a subgroup of $G$ such that $F(X)/H$ is not terminal. The goal of the GAP code below is to compute the second Betti number of a terminalization of $F(X)/H$, which by Theorem \ref{thm:b2} is given by the formula
$$
b_2(Y)= \rk(H^2(F(X),\Z)^H) + n_2 + n_{31} + 2n_{32}.
$$

The code proceeds as follows. It imports \cite[Table 12]{HoehnMason2019}, where the possible ranks of the coinvariant lattice $H^2(X',\Z)_{G'}$ -- the orthogonal complement of $H^2(X',\Z)^{G'}$ -- are determined, for any finite group $G'$ of symplectic automorphisms of a hyperkahler manifold $X'$ of $\text{K}3^{[2]}$-type. In particular, each entry in \texttt{rkL} lists
\begin{itemize}
    \item the row number $i$ in \cite[Table 12]{HoehnMason2019}, 
    \item the group ID of the corresponding group $G'_i$,
    \item the rank of the coinvariant lattice  $H^2(X',\Z)_{G'_i}$.
\end{itemize}

Given the group $G$ and the list \texttt{output} from Appendix \ref{app:invariants}, the code matches each subgroup $H$ in \texttt{output} with the entries in \texttt{rkL} with the same group ID of $H$, and compute the corresponding $b_2$. For the groups appearing multiple times in \texttt{rkL} with different rank, it remains to study the correct rank of $H^2(F(X),\Z)^H$ to determine the correct $b_2$; see Section \ref{subsec:rk}.
\bigskip

\begin{lstlisting}[language=GAP]
#Import from [HM19] the ranks on the coinvariant lattices H^2(X',Z)_{G'_i} 
\end{lstlisting}
\begin{multicols}{3}
\begin{lstlisting}[language=GAP]
rkL:=[
 [1 ,[1, 1], 0]
,[2 ,[2, 1], 8]
,[3 ,[3 ,1 ], 12]
,[4 ,[3 ,1 ], 18]
,[5 ,[4 ,1 ], 14]
,[6 ,[4 ,2 ], 12]
,[7 ,[5 ,1 ], 16]
,[8 ,[6 ,2 ], 16]
,[9 ,[6 ,1 ], 14]
,[10 ,[6 ,2 ], 18]
,[11 ,[7 ,1 ], 18]
,[12 ,[8 ,1 ], 18]
,[13 ,[8 ,3 ], 15]
,[14 ,[8 ,2 ], 16]
,[15/17 ,[8 ,4 ], 17]
,[16 ,[8 ,5 ], 14]
,[18 ,[9 ,2 ], 18]
,[19 ,[9 ,2 ], 16]
,[20 ,[9 ,1 ], 20] 
,[21 ,[9 ,2 ], 20] 
,[22 ,[10 ,1 ], 16]
,[23 ,[11 ,1 ], 20] 
,[24 ,[12 ,4 ], 16]
,[25 ,[12 ,3 ], 16]
,[26 ,[12 ,1 ], 19]
,[27/29 ,[12 ,5], 18]
,[28 ,[12 ,2 ], 20] 
,[30 ,[12 ,1 ], 18]
,[31 ,[14 ,2 ], 20] 
,[32 ,[15 ,1 ], 20] 
,[33/44 ,[16 ,6 ], 19]
,[34/38 ,[16 ,2], 18]
,[35/37 ,[16 ,14], 15]
,[36 ,[16 ,10], 17]
,[39/42/43,[16 ,12 ], 18]
,[40 ,[16 ,8 ], 18]
,[41 ,[16 ,3 ], 17]
,[45 ,[16 ,9 ], 19]
,[46 ,[16 ,13 ], 17]
,[47 ,[16 ,11 ], 16]
,[48/49 ,[18 ,3 ], 18]
,[50 ,[18 ,4 ], 16]
,[51 ,[18 ,3 ], 20] 
,[52 ,[18 ,5 ], 20] 
,[53 ,[20 ,3 ], 18]
,[54 ,[21 ,1 ], 18]
,[55/57 ,[24 ,3 ], 19]
,[56 ,[24 ,12 ], 17]
,[58 ,[24 ,13 ], 18]
,[59 ,[24 ,4 ], 20] 
,[60 ,[24 ,8 ], 18]
,[61 ,[27 ,3 ], 18]
,[62 ,[27 ,5 ], 20] 
,[63/64 ,[27 ,4 ], 20]
,[65 ,[27 ,2 ], 20] 
,[66 ,[27 ,5 ], 18]
,[67 ,[30 ,2 ], 20] 
,[68 ,[32 ,49 ], 17]
,[69 ,[32 ,11 ], 19]
,[70/73 ,[32 ,8], 20] 
,[71 ,[32 ,44 ], 19]
,[72 ,[32 ,31 ], 18]
,[74 ,[32 ,50], 18]
,[75 ,[32 ,7 ], 19]
,[76 ,[32 ,6 ], 18]
,[77 ,[32 ,27 ], 17]
,[78 ,[36 ,11 ], 18]
,[79 ,[36 ,9 ], 18]
,[80 ,[36 ,12 ], 20] 
,[81 ,[36 ,10 ], 18]
,[82 ,[36 ,7 ], 20] 
,[83 ,[42 ,2 ], 20] 
,[84 ,[48 ,30 ], 19]
,[85 ,[48 ,50 ], 17]
,[86 ,[48 ,29 ], 19]
,[87 ,[48 ,3 ], 18]
,[88 ,[48 ,48 ], 18]
,[89/90 ,[48 ,32 ], 20] 
,[91/94 ,[48 ,28 ], 20] 
,[92/93 ,[48 ,49 ], 19]
,[95 ,[54 ,12 ], 20] 
,[96 ,[54 ,13 ], 18]
,[97 ,[54 ,8 ], 18]
,[98 ,[54 ,13 ], 20] 
,[99 ,[55 ,1 ], 20] 
,[100,[ 56 ,11 ] , 20] 
,[101,[ 60 ,7 ] , 20] 
,[102,[ 60 ,5 ] , 18]
,[103 ,[64 ,242 ] , 18]
,[104 ,[64 ,32 ] , 19]
,[105 ,[64 ,35 ] , 19]
,[106 ,[64 ,136 ] , 19]
,[107 ,[64 ,36] , 20] 
,[108 ,[64 ,138 ] , 18]
,[109 ,[72 ,39 ] , 20]
,[110 ,[72 ,41 ] , 19]
,[111 ,[72 ,40 ] , 19]
,[112 ,[72 ,43 ] , 18]
,[113 ,[72 ,22 ] , 20] 
,[114 ,[80 ,49 ] , 19]
,[115 ,[81 ,13 ] , 20] 
,[116 ,[81 ,8 ] , 20] 
,[117 ,[81 ,7 ] , 20] 
,[118 ,[81 ,15 ] , 20] 
,[119 ,[81 ,12 ] , 18]
,[120 ,[96 ,227 ] , 18]
,[121 ,[96 ,195 ] , 19]
,[122 ,[96 ,204 ] , 19]
,[123 ,[96 ,190 ] , 20]
,[124 ,[96 ,64 ] , 19]
,[125 ,[96 ,70 ] , 19]
,[126 ,[108 ,38 ] , 20] 
,[127 ,[108 ,15 ] , 20] 
,[128 ,[108 ,40 ] , 20]
,[129 ,[108 ,37 ] , 20] 
,[130 ,[108 ,37 ] , 19]
,[131 ,[108 ,36 ] , 20] 
,[132 ,[120 ,34 ] , 19]
,[133 ,[128 ,931 ] , 19]
,[134 ,[144 ,184 ] , 19]
,[135 ,[144, 0] , 20] 
,[136 ,[160 ,234 ] , 19]
,[137 ,[162 ,10 ] , 20] 
,[138 ,[162 ,46 ] , 18]
,[139 ,[162 ,52 ] , 20] 
,[140 ,[168 ,42 ] , 19]
,[141 ,[168 ,43 ] , 20] 
,[142 ,[180 ,19 ] , 20] 
,[143 ,[192 ,1493 ] , 19]
,[144 ,[192 ,1024 ] , 20] 
,[145 ,[192 ,201 ] , 20] 
,[146 ,[192 ,1009 ] , 20]
,[147 ,[192 ,184 ] , 20]
,[148 ,[192 ,955 ] , 19]
,[149 ,[192 ,1492 ] , 20]
,[150 ,[192 ,1023 ] , 18]
,[151 ,[216 ,161 ] , 20] 
,[152 ,[216 ,158 ] , 20] 
,[153 ,[243 ,57 ] , 20] 
,[154 ,[243 ,65 ] , 18]
,[155 ,[243 ,51 ] , 20] 
,[156 ,[288 ,1025 ] , 20]
,[157 ,[288 ,1026 ] , 19]
,[158 ,[320 ,1635 ] , 20] 
,[159 ,[324 ,167 ] , 20] 
,[160 ,[324 ,160 ] , 20] 
,[161 ,[324 ,163 ] , 20] 
,[162 ,[336 ,0] , 20] 
,[163 ,[360 ,118 ] , 19]
,[164 ,[360 ,120 ] , 20]
,[165 ,[384 ,5603 ] , 20] 
,[166 ,[384 ,5678 ] , 20]
,[167 ,[384 ,18133 ] , 20] 
,[168 ,[384 ,18135 ] , 19]
,[169 ,[405 ,15 ] , 20] 
,[170 ,[486 ,249 ] , 18]
,[171 ,[486 ,166 ] , 20] 
,[172 ,[576 ,8652 ] , 20] 
,[173 ,[576 ,0] , 20] 
,[174 ,[576 ,5129 ] , 20] 
,[175 ,[648 ,722 ] , 20] 
,[176 ,[648 ,704 ] , 20] 
,[177 ,[660 ,13 ] , 20] 
,[178 ,[720 ,765 ] , 20] 
,[179 ,[720 ,763 ] , 20]
,[180 ,[729 ,321 ] , 20] 
,[181 ,[810 ,101 ] , 20] 
,[182 ,[960 ,11358 ] , 20] 
,[183 ,[960 ,11357 ] , 19]
,[184 ,[972 ,877 ] , 20] 
,[185 ,[972 ,776 ] , 19]
,[186 ,[972 ,777 ] , 20] 
,[187 ,[1152 ,0] , 20] 
,[188 ,[1344 ,0] , 20] 
,[189 ,[1458 ,1229 ] , 20] 
,[190 ,[1920 ,0] , 20] 
,[191 ,[1944 ,3559 ] , 20] 
,[192 ,[1944 ,3877 ] , 20] 
,[193 ,[2520 ,0] , 20] 
,[194 ,[2916 ,0] , 20] 
,[195 ,[4860 ,0] , 20] 
,[196 ,[5760 ,0] , 20] 
,[197 ,[20160 ,0] , 20] 
,[198 ,[29160 ,0] , 20]];;
\end{lstlisting}
\end{multicols}

\begin{lstlisting}[language=GAP]
#Create a list table that, for each subgroup H in cc, gives an entry with [#H in cc, #rk(L_H) in rkL, IdGroup(H) or Size(H),[rk(L_H),N_3],n_2,[n_3,n_31,n_32],b_2,IdGroup(pai)]: 

table:=[];;
for i in output do
    for j in rkL do
        if i[2]=j[2] then
            Add(table, [i[1],j[1],i[2],[j[3],i[4]],i[3],
            [i[5],i[6][1],i[6][2]],23-j[3]+i[3]+i[6][1]+2*i[6][2],i[7]]);
        fi;
    od;
od;
\end{lstlisting}

\subsection{Numerical obstruction to deformation equivalence}\label{app:obstruction}
Let $Y$ be a terminalization of a quotient $F(X)/G$ as constructed in Section \ref{sec:term}, and assume that $Y$ has simply connected regular locus. By \cite[Proposition 3.19 and 3.21]{Menet2022}, if $Y$ is deformation equivalent to a Fujiki fourfold $S(H)_\theta^{[2]}$ or to a terminalization of a symplectic quotient $S^{[2]}/H$, for a $\textrm{K3}$ surface $S$, then $$\left(\tfrac{|G|}{|H|}\right)^{\tfrac{1}{2}}$$ is rational.
Similarly, if $Y$ is deformation equivalent to a terminalization of a symplectic quotient $\textrm{Kum}_2(A)/H$ for an abelian surface $A$, then 
$$\left(\tfrac{|G|}{3|H|}\right)^{\tfrac{1}{2}}$$ is rational.

The GAP code below compares the terminalizations in Table \ref{table:term} with the Fujiki fourfolds studied in \cite{Menet2022}, and the terminalizations constructed in \cite{BertiniGrossiMauriMazzon2024}. Whenever the second Betti numbers coincide, it checks whether the rationality conditions above are satisfied. As input, the code takes  
\begin{itemize}
    \item a list \texttt{list}, whose entries are of the form $[\text{IdGroup}(G), b_2(Y), |\Aut^s(X)|]$, one for each terminalization $Y$ in Table \ref{table:term} with simply connected regular locus;
    \item the lists \texttt{vFuj}, \texttt{vK3} and \texttt{vKum} which, for a given value $b$, contain the cardinalities of the groups $H$ such that the corresponding Fujiki variety $S(H)_\theta^{[2]}$, terminalization of $S^{[2]}/H$, or of $\textrm{Kum}_2(A)/H$, respectively, has second Betti number equal to $b$.
\end{itemize}
The program produces the lists \texttt{listFuj}, \texttt{listK3}, and \texttt{listKum}, consisting of all terminalizations $Y$ which either have a different second Betti number or share the same $b_2$ but fail the numerical condition above when compared with, respectively, any Fujiki fourfold, any terminalization of symplectic quotients of a Hilbert square, or any terminalization of Kummer fourfolds from \cite{Menet2022,BertiniGrossiMauriMazzon2024}. Finally the list \texttt{listnew} contains the terminalizations $Y$ which certainly represent new deformation classes.

\begin{lstlisting}[language=GAP]
#List of [GroupID(G), b_2(Y), |Aut^s(X)|], for any terminalization Y of F(X)/G with simply connected regular locus
\end{lstlisting}
\begin{multicols}{3}
\begin{lstlisting}[language=GAP]
list:=[ [ 
[ 2, 1 ], 16, 360 ], 
[ [ 4, 2 ], 14, 360 ], 
[ [ 6, 1 ], 10, 360 ], 
[ [ 8, 3 ], 11, 360 ], 
[ [ 10, 1 ], 8, 360 ], 
[ [ 12, 4 ], 10, 360 ], 
[ [ 18, 4 ], 8, 360 ], 
[ [ 24, 12 ], 8, 360 ], 
[ [ 24, 8 ], 8, 360 ],
[ [ 36, 10 ], 8, 360 ], 
[ [ 60, 5 ], 6, 360 ], 
[ [ 72, 43 ], 7, 360 ], 
[ [ 120, 34 ], 6, 360 ], 
[ [ 360, 120 ], 5, 360 ], 
[ [ 2, 1 ], 16, 660 ], 
[ [ 4, 2 ], 14, 660 ], 
[ [ 6, 1 ], 10, 660 ], 
[ [ 10, 1 ], 8, 660 ], 
[ [ 12, 4 ], 10, 660 ], 
[ [ 60, 5 ], 6, 660 ], 
[ [ 660, 13 ], 4, 660 ], 
[ [ 2, 1 ], 16, 720 ], 
[ [ 4, 2 ], 14, 720 ], 
[ [ 6, 1 ], 10, 720 ], 
[ [ 8, 3 ], 11, 720 ], 
[ [ 10, 1 ], 8, 720 ], 
[ [ 18, 4 ], 8, 720 ], 
[ [ 24, 12 ], 8, 720 ], 
[ [ 60, 5 ], 6, 720 ], 
[ [ 360, 118 ], 5, 720 ], 
[ [ 2, 1 ], 16, 1944 ], 
[ [ 3, 1 ], 7, 1944 ], 
[ [ 6, 2 ], 8, 1944 ], 
[ [ 6, 1 ], 10, 1944 ], 
[ [ 18, 4 ], 8, 1944 ], 
[ [ 18, 3 ], 8, 1944 ], 
[ [ 54, 13 ], 8, 1944 ], 
[ [ 54, 8 ], 8, 1944 ], 
[ [ 162, 46 ], 8, 1944 ], 
[ [ 486, 249 ], 8, 1944 ], 
[ [ 2, 1 ], 16, 2520 ], 
[ [ 4, 2 ], 14, 2520 ], 
[ [ 6, 1 ], 10, 2520 ], 
[ [ 8, 3 ], 11, 2520 ], 
[ [ 10, 1 ], 8, 2520 ], 
[ [ 12, 4 ], 10, 2520 ], 
[ [ 18, 4 ], 8, 2520 ], 
[ [ 24, 12 ], 8, 2520 ], 
[ [ 24, 8 ], 8, 2520 ], 
[ [ 60, 5 ], 6, 2520 ], 
[ [ 72, 43 ], 7, 2520 ], 
[ [ 120, 34 ], 6, 2520 ], 
[ [ 168, 42 ], 5, 2520 ], 
[ [ 360, 118 ], 5, 2520 ], 
[ [ 2520, 0 ], 4, 2520 ], 
[ [ 2, 1 ], 16, 29160 ], 
[ [ 3, 1 ], 7, 29160 ], 
[ [ 4, 2 ], 14, 29160 ],
[ [ 6, 1 ], 10, 29160 ], 
[ [ 6, 2 ], 8, 29160], 
[ [ 8, 3 ], 11, 29160], 
[ [ 9, 2 ], 7, 29160], 
[ [ 10, 1 ], 8, 29160], 
[ [ 12, 4 ], 10, 29160], 
[ [ 18, 4 ], 8, 29160], 
[ [ 18, 3 ], 8, 29160], 
[ [ 18, 5 ], 8, 29160], 
[ [ 24, 12 ], 8, 29160], 
[ [ 24, 8 ], 8, 29160], 
[ [ 27, 5 ], 11, 29160], 
[ [ 36, 10 ], 8, 29160], 
[ [ 36, 12 ], 8, 29160], 
[ [ 54, 13 ], 8, 29160], 
[ [ 54, 13 ], 8, 29160], 
[ [ 54, 8 ], 8, 29160], 
[ [ 54, 12 ], 10, 29160], 
[ [ 60, 5 ], 6, 29160], 
[ [ 72, 43 ], 7, 29160], 
[ [ 72, 22 ], 7, 29160], 
[ [ 72, 40 ], 7, 29160], 
[ [ 81, 15 ], 23, 29160], 
[ [ 108, 40 ], 8, 29160], 
[ [ 108, 38 ], 10, 29160], 
[ [ 162, 52 ], 16, 29160], 
[ [ 162, 46 ], 8, 29160], 
[ [ 162, 10 ], 8, 29160], 
[ [ 216, 158 ], 8, 29160], 
[ [ 324, 167 ],14, 29160], 
[ [ 360, 118 ],5, 29160], 
[ [ 486, 249 ],8, 29160], 
[ [ 486, 166 ],10, 29160], 
[ [ 648, 704 ],7, 29160], 
[ [ 648, 722 ],11,29160], 
[ [ 810, 101 ],8,29160], 
[ [ 1458, 1229 ],8,29160], 
[ [ 1944, 3877 ],8,29160], 
[ [ 4860, 0 ],6,29160], 
[ [ 29160, 0 ],5,29160], 
[ [ 2, 1 ], 16, 48 ], 
[ [ 4, 2 ], 14, 48 ], 
[ [ 6, 1 ], 10, 48 ], 
[ [ 8, 3 ], 11, 48 ], 
[ [ 12, 4 ], 10, 48 ], 
[ [ 48, 29 ], 6, 48 ] ];;
\end{lstlisting}
\end{multicols}
\begin{lstlisting}[language=GAP]
#List of [b,[|H_1|,|H_2|,..]], which, for a given value b, contains the cardinalities of the groups H_i such that the corresponding Fujiki variety S(H)_\theta^{[2]} has second Betti number equal to b
vFuj:=[[4,[144]],[5,[36,48,96]],[6,[18,24,36,48,96]],[7,[9,12,48,72]],
    [8,[6,12,16,18,24,36]],[10,[4,6,8,12,48]],[11,[3,8,32]],[14,[4]],[16,[2]] ];;
    
#List of [b,[|H_1|,|H_2|,..]], which, for a given value b, contains the cardinalities of the groups H_i such that the corresponding terminalization of S^{[2]}/H has second Betti number equal to b
vK3:=[[5,[168,360]],[6,[60,120]],[8,[10]]];;

#List of [b,[|H_1|,|H_2|,..]], which, for a given value b, contains the cardinalities of the groups H_i such that the corresponding terminalization of Kum_2(A)/H has second Betti number equal to b 
vKum:=[[7,[3,24,216,1944]],[8,[2,6,18,54,162,486]],[10,[18,162]],[11,[9]],
    [16,[54]],[23,[27]]];;

listFuj:=[];;
for i in list do
    b:=0;
    for j in vFuj do
        if i[2]=j[1] then
            b:=b+1;
            a:=0;
            for h in j[2] do
                q:=i[1][1]*h^-1;
                s:=ER(q);
                if IsRat(s)=true then
                    a:=1;
                fi;
            od;
            if a=0 then
                Add(listFuj,i);
            fi;
        fi;
    od;
    if b=0 then
        Add(listFuj,i);
    fi;
od;
listK3:=[];;
for i in list do
    b:=0;
    for j in vK3 do
        if i[2]=j[1] then
            a:=0;
            b:=b+1;
            for h in j[2] do
                q:=i[1][1]*h^-1;
                s:=ER(q);
                if IsRat(s)=true then
                    a:=1;
                fi;
            od;
            if a=0 then
                Add(listK3,i);
            fi;
        fi;
    od;
    if b=0 then
        Add(listK3,i);
    fi;
od;
listKum:=[];;
for i in list do
    b:=0;
    for j in vKum do
        if i[2]=j[1] then
            a:=0;
            b:=b+1;
            for h in j[2] do
		        f:=1/3;
                q:=f*i[1][1]*h^-1;
                s:=ER(q);
                if IsRat(s)=true then
                    a:=1;
                fi;
            od;
            if a=0 then
                Add(listKum,i);
            fi;
        fi;
    od;
    if b=0 then
        Add(listKum,i);
    fi;
od;
listnew1:=Intersection(listFuj,listK3);;
listnew:=Intersection(listnew1,listKum);
\end{lstlisting}

\printbibliography

\end{document}